\documentclass[11pt,a4paper,leqno]{article}
\usepackage{amssymb}
\usepackage{amsmath}
\usepackage{amsthm}
\usepackage{mathtools}
\usepackage{dsfont}
\usepackage{graphicx,epsfig}
\usepackage{float}
\usepackage{color}
\usepackage{capt-of}
\usepackage{xcolor}
\usepackage{capt-of}
\usepackage{comment}

 \usepackage[shortlabels]{enumitem}

\usepackage{fancyhdr}
\usepackage{hyperref}
\usepackage{eucal}
\usepackage{pgfplots}
\usepackage{subcaption}
\usepackage{multirow} 
\usepackage{slashbox} 

\pgfplotsset{width=5.5cm,compat=1.9}

\usepackage{xargs}
\usepackage[colorinlistoftodos,textsize=small]{todonotes}
\newcommandx{\unsure}[2][1=]{\todo[linecolor=red,backgroundcolor=red!25,bordercolor=red,#1]{#2}}
\newcommandx{\change}[2][1=]{\todo[linecolor=blue,backgroundcolor=blue!25,bordercolor=blue,#1]{#2}}
\newcommandx{\info}[2][1=]{\todo[linecolor=green,backgroundcolor=green!25,bordercolor=green,#1]{#2}}
\newcommandx{\improvement}[2][1=]{\todo[linecolor=yellow,backgroundcolor=yellow!25,bordercolor=yellow,#1]{#2}}

\usepackage{amsthm,amssymb,amsmath,amsfonts}
\usepackage{appendix}

\usepackage{pgf, tikz}
\usetikzlibrary{arrows,shapes,positioning}

\newtheorem{theorem}{\textbf{Theorem}}[section]
\newtheorem{remark}[theorem]{\textbf{Remark}}

\newtheorem{lemma}[theorem]{\textbf{Lemma}}

\newtheorem{acknowledgement}{\textbf{Acknowledgement}}

\numberwithin{equation}{section}

\usepackage{titlesec}
\titleformat\section{}{}{0pt}{\Large\scshape\filcenter\thesection{} - }

\def\aa{\alpha}

\def\ll{\lambda}

\def\O{\Omega}

\def\th{\theta}

\def\ee{\varepsilon}

\newcommand{\RR}{ \mathbb{R}}

\newcommand{\ZZ}{ \mathbb{Z}}

\newcommand{\EE}{{\mathbb E}}

\newcommand{\cF}{{\mathcal F}}

\newcommand{\cG}{{\mathcal G}}

\newcommand{\cL}{{\mathcal L}}

\newcommand{\cP}{{\mathcal P}}
\newcommand{\cR}{{\mathcal R}}

\newcommand{\cT}{{\mathcal T}}

\newcommand{\cV}{{\mathcal V}}

\newcommand{\cZ}{{\mathcal Z}}

\newcommand{\ati}{{\widetilde a}}

\newcommand{\ft}{{\widetilde f}}

\newcommand{\Vt}{{\widetilde V}}

\newcommand{\ptt}{{\partial_{t}}}

\newcommand{\norminf}[2][]{{\left\| #2 \right\|_{\infty}^{#1}}}

\DeclareMathOperator{\divo} {div}

\newcommand{\lp}{\left(}
\newcommand{\rp}{\right)}
\newcommand{\lc}{\left\{}
\newcommand{\rc}{\right\}}
\newcommand{\lb}{\left[}
\newcommand{\rb}{\right]}

\newcommand{\labs}{\left|}
\newcommand{\rabs}{\right|}

\newcommand{\ds}{\displaystyle}

\title{Mean Field Games of Controls: Finite Difference Approximations}
\author{Y. Achdou \thanks{Univ. Paris Diderot, Sorbonne Paris Cit{\'e}, Laboratoire Jacques-Louis Lions, UMR 7598, UPMC, CNRS, F-75205 Paris, France.
 achdou@ljll.univ-paris-diderot.fr}, Z. Kobeissi \thanks {Univ. Paris Diderot, Sorbonne Paris Cit{\'e}, Laboratoire Jacques-Louis Lions, UMR 7598, UPMC, CNRS, F-75205 Paris, France.
 zkobeissi@math.univ-paris-diderot.fr}}

\begin{document}
\maketitle

\begin{abstract}
  We consider a class of mean field games in which the agents interact through both their states and controls,
  and we focus on situations in which a generic agent tries to adjust her speed (control) to an average speed (the average is made in a neighborhood in the state space). In such cases, the monotonicity assumptions that are frequently made in the theory of mean field games do not hold, and uniqueness cannot be expected in general.
  
  Such model lead to systems of forward-backward nonlinear nonlocal parabolic equations; the latter are supplemented with various kinds of boundary conditions, in particular Neumann-like boundary conditions stemming from  reflection conditions on the underlying controled stochastic processes.

  The present work deals with numerical approximations of the above mentioned systems.  After describing the finite difference scheme, we propose an iterative method  for solving the systems of nonlinear equations that arise in the discrete setting; it combines a continuation method, Newton iterations and inner loops of a bigradient like solver.

 The numerical method is used for simulating two examples. We also make experiments on the  behaviour of the iterative algorithm when the parameters of the model vary.
  
\end{abstract}
\section{Introduction}
\label{sec:introduction}
The theory of  mean field games, ({\sl MFGs} for short), aims at studying deterministic or stochastic differential
games (Nash equilibria) as the number of agents tends to infinity. It supposes that the rational agents are  indistinguishable and  individually have a negligible influence on the game, and that each individual strategy is influenced by some averages of quantities depending on the states (or the controls as in the present work) of the other agents.  MFGs have been introduced in the pioneering  works of J-M. Lasry and P-L. Lions \cite{MR2269875,MR2271747,MR2295621}. Independently and at approximately  the same time, the notion of mean field games arose in the engineering literature, see the works of  M.Y. Huang, P.E. Caines and R.Malham{\'e} \cite{MR2352434,MR2346927}. 

The present work deals with numerical approximations of mean field games in which the agents interact through both their states and controls; it follows a more theoretical work by the second author, \cite{2019arXiv190411292K}, which is devoted to the mathematical analysis of the related systems of nonlocal partial differential equations.   There is not much  literature on MFGs  in which the agents also interact through their controls, see
 \cite{MR3160525,MR3112690,MR3805247,MR3752669,2019arXiv190205461F,2019arXiv190411292K}.
To stress the fact that the latter situation is considered,  we will sometimes use the terminology {\sl mean field games of control} and the acronym {\sl MFGC}.

The MFGC that is considered in the present work is described by the following system of nonlocal partial differential equations:
\begin{equation}
    \label{eq:MFGC}
    \lc
    \begin{aligned}
            &-\ptt u
            - \nu\Delta u
            +H\lp x,\nabla_xu(t,x),\mu(t)\rp
            = f(x,m(t))
            &\text{ in }
            [0,T)\times\Omega,
            \\
            & \ptt m
            - \nu\Delta m
            -\divo\lp H_p\lp x,\nabla_xu(t,x),\mu(t)\rp m\rp
            =0
            &\text{ in }
            (0,T]\times\Omega,
            \\
            &\mu(t)
            = \Bigl(
            I_d,
            -H_p\lp \cdot,\nabla_xu(t,\cdot),\mu(t)\rp
            \Bigr){\#}m(t)
            &\text{ in } [0,T],
            \\
            &u(T,x)
            =\phi(x)
            &\text{ in }
            \Omega,\\
            &m(0,x)
            =m_0(x)
            &\text{ in }
            \Omega.
        \end{aligned}
        \right.
\end{equation}
Here, the state space $\Omega$ is a bounded domain of $\RR^d$ with a piecewise smooth boundary, while the controls are vectors of $\RR^d$. 
  The time horizon is $T>0$, and the parameter $\nu>0$ is linked to the level of noise in the trajectories (the dynamics of a given agent is described by 
$dX_t= \alpha_t dt +\sqrt{2\nu} dB_t$ where $(B_t)$ is a standard $d$-dimensional Brownian motion and $\alpha_t\in \RR^d$ is the control at time $t$).  A special feature of the present model is that the third argument of the Hamiltonian $H$ is a probability measure  $\mu(t)\in \cP\lp\Omega\times\RR^d\rp$, which stands for the joint probability  of the states and optimal controls of the agents at time $t$.
In (\ref{eq:MFGC}),   $u: [0,T]\times \Omega \to \RR$ and  $m : [0,T]\times \Omega \to \RR_+$  respectively stand for 
the value function of a representative agent and the distribution of states. The first, respectively second line in (\ref{eq:MFGC}) is the  Hamilton-Jacobi-Bellman equation ({\sl HJB} for short) which leads to  the optimal control of a representative agent, respectively  the Fokker-Planck-Kolmogorov equation ({\sl FP}  for short) which describes the transport-diffusion of the distribution of states by the optimal control law. The HJB equation is a backward w.r.t. time parabolic equation and is supplemented with a condition at $t=T$ which involves the terminal cost function $\phi: \Omega\to \RR$, whereas the FP equation is forward w.r.t. time and is supplemented with an initial condition, which translates the fact that the initial distribution of states is known. The HJB equation also involves the so-called {\sl coupling } function $f:\Omega\times \RR_+\to \RR$,   or, in other words,  $f(X_t, m(t, X(t)))$ is part of the running cost of a representative agent. We shall discuss later the boundary conditions on $(0,T)\times \partial \Omega$ associated with the  first two equations in (\ref{eq:MFGC}).
  The third equation in (\ref{eq:MFGC}) gives the connection between $\mu(t)$ and $m(t)$, and shows in particular that $\mu(t)$ is supported by a $d$-dimensional geometrical object, namely the graph of the feedback law: $\Omega\to \RR^d$, $x\mapsto - H_p (x, \nabla u (t, x),\mu(t))$. 

\subsection{A brief discussion on the mathematical analysis of (\ref{eq:MFGC})}\label{sec:brief-disc-math}
Recall that the Hamiltonian of the problem is $(x,p,\mu)\mapsto H(x,p,\mu)$, $(x,p,\mu)\in \Omega\times \RR^d \times \cP\lp\Omega\times\RR^d\rp$.

From the viewpoint of mathematical analysis, a priori estimates for (\ref{eq:MFGC}) are more difficult to obtain than in the case when the agents interact only via the distribution of states $m$. Indeed, in the latter case, if for example the costs $f$ and $\phi$ are uniformly bounded, then a priori estimates on $\norminf{u}$ stem from the maximum principle for second-order parabolic equations. By contrast, since the Hamiltonian in (\ref{eq:MFGC}) depends 
non-locally on $\nabla_x u$, the maximum principle applied to the HJB equation only permits to bound  $\norminf{u}$ by a quantity  which depends (quadratically under standard assumptions on $H$) on $\norminf{\nabla_xu}$, and this information may be useless without additional arguments.

  If the agents interact only through the distribution of states and if the Hamiltonian depends separately on $p$ and $m$, a natural assumption  is that the latter is monotone with respect to $m$, see \cite{MR2269875,MR2271747,MR2295621}; it implies existence and uniqueness of solutions, see in \cite{Lions_video,MR2295621}. Such an assumption is quite sensible in many situations, since  it models the  aversion of the agents to highly crowded regions of the state space. It is possible to extend these arguments to MFGCs, see \cite{MR3752669} for a probabilistic point of view and \cite{MR3805247} for a PDE point of view,  and the monotonicity assumption then means that the agents favor controls that are opposite to the main stream. In \cite{2019arXiv190411292K} and in the present work, we prefer to avoid such an  assumption, because it is generally not satisfied, at least in models of crowd motion: indeed, in models of traffic or pedestrian flows, a generic agent would rather try to adjust her speed (control) to the average speed in a neighborhood  of her position. 

 The third equation in (\ref{eq:MFGC}) can be seen as a fixed point problem for $\mu$ given $u$ and $m$, which turns to be well-posed under  
the Lasry-Lions monotonicity assumption adapted to MFGC, provided that $u$ and $m$ are smooth enough.  We shall replace this assumption by a 
new structural condition which has been introduced in  \cite{2019arXiv190411292K},  namely that $H_p$ depends linearly on the variable $\mu$ and is a contraction with respect to  $\mu$  (using a suitable distance on probability measures). In the context of crowd motion, this structural condition 
is satisfied if the representative agent targets controls that are proportional to an average of the controls chosen by the other agents nearby, with a positive proportionality coefficient smaller than one. Were this coefficient equal to or larger than one,  it would be  easy to cook up examples 
in which there is no solution to (\ref{eq:MFGC}) or even to the $N$-agent game, see Remark  \ref{sec:simpl-one-shot-1} below.

In \cite{2019arXiv190411292K}, the focus is put on the existence of solutions rather than on  their 
uniqueness. Indeed, without the monotonicity condition, uniqueness  is unlikely in general if $T$ is not small. Consider for example 
a game in which the   function   $\phi$ has   two perfectly symmetrical minima (the targets), $f=0$ and where the initial distribution of states 
has the same symmetry. If $H$ depends on $|\nabla u(t,x)|$ only (no interaction through the controls),
 then a representative agent will simply travel to the minimum which is  closest to her initial position. On the contrary, if the representative agent  favors a control close to the average one, then there are at least two symmetrical solutions in which the whole distribution moves toward one of the two minima. 

Going back to existence results, it is now frequent in the MFG literature to obtain energy estimates by testing  the HJB equation by $m$,  the FP equation by $u$, summing the resulting equations, integrating in the time and state variables then making suitable integrations by parts. If the value function is uniformly bounded from below (which is often the case even if there are  interactions through the controls),
this results in a relationship between the $L^{\infty}\lp [0,T];L^1\lp\Omega\rp\rp$-norm
of the positive part of $u$ and the $L_m^2\lp[0,T]\times\Omega;\RR^d\rp$-norm of $\nabla_xu$;
this observation can then be used to obtain additional a priori estimates, which may be  combined with the ones  obtained from the maximum principle
and discussed above, and finally with Bernstein method.  This strategy has been implemented in \cite{2019arXiv190411292K} and leads to the existence of a solution to (\ref{eq:MFGC}) under suitable assumptions. By and large, existence was proved in \cite{2019arXiv190411292K} for periodic problems in any of the following cases:
\begin{itemize}
    \item short time horizon
    \item monotonicity 
    \item small enough parameters (in particular the contraction factor mentioned above, see the parameters $\lambda$ and $\theta$ in (\ref{eq:HJB})-(\ref{eq:defZ}) below)
    \item weak dependency of the average drift  on the state variable
\end{itemize}
Note that  in \cite{carmona2015}, existence and uniqueness have been proved with probabilistic arguments in the case where the Hamiltonian depends separately on $p$ and $\mu$.

\subsection{A more detailed description of the considered class of MFGCs}\label{sec:more-details-h}
For any $x\in \partial \Omega$, let $n(x)$ be the outward pointing unit normal vector to $\partial \Omega$ at $x$.
The dynamics  of a representative agent is given by 
\begin{equation}
    dX_t    =    \aa_tdt    +\sqrt{2\nu}dB_t -  2\nu n(X_t) dL_t,\quad  X_t\in \overline \Omega,\quad    0\leq t \leq T,
\end{equation}
where $\lp B_t\rp_{t\in\lb 0,T\rb}$ is a  standard $d$-dimensional Brownian motion,  $(\aa_t)_{t\in [0,T]}$ is the  control, a stochastic process adapted to $(B_t)$ with values in $\RR^d$,  $L_t$ is the local time of the process $X_t$ on $\partial \Omega$.
We assume that $X_{|t=0}$ is a random variable in $\overline \Omega$, independent of $(B_t)$ for all $t>0$,
whose law is  $ \cL\lp X_{|t=0}\rp=m_0$. 
\\
In what follows, the part of the running cost which models the  interactions via the controls will involve
 an average drift  $V\in \RR ^d$:
\begin{equation}
    V(t,x)
    =\frac 1 {Z(t,x) }
    \int_{(y,\alpha) \in \Omega\times\RR^d}
    \aa K(x,y)d\mu(t,y,\aa),
\end{equation}
for $(t,x)\in[0,T]\times\Omega$, where $K$ is a nonnegative kernel, $Z(t,x)=\int_{\Omega}K(x,y)dm(t,y)$
is a normalization factor, $\mu(t,\cdot,\cdot)$ is the law of the joint distribution of the states and the controls, and $m(t,\cdot)$ is the law of the distribution of states. 
\\
Hereafter, we are going to focus in MFGCs in which the cost to be minimized by a representative agent is
\begin{equation}
    \label{eq:cost}
    J(\aa)
    =
    \EE\lb
    \int_0^T \Bigl(
    \frac{\th}2\labs\aa_t-\ll V(t,X_t)\rabs^2
    +\frac{1-\th}2\labs\aa_t\rabs^2
    +cm(t,X_t) +f_0(X_t)\Bigr) dt
    +\phi(X_T)
    \rb,
\end{equation}
where $\ll,\th$ are  real numbers such that
$-1<\ll<1$ and $0<\th\leq 1$. This leads us to define the Lagrangian
\begin{equation}
  \label{eq:1}
    L(\aa,V)
    =
    \frac{\th}2\labs\aa-\ll V\rabs^2
    +\frac{1-\th}2\labs\aa\rabs^2,
    \qquad
    \lp \aa,V\rp\in\RR^d\times\RR^d,
\end{equation}
which is convex with respect to $\alpha$, and its Legendre transform (with respect to $\alpha$):
\begin{equation}
  \label{eq:2}
    H(p,V)
    =
    \frac12\labs p-\ll\th V\rabs^2
    -\frac{\ll^2\th}2\labs V\rabs^2,
    \qquad
    \lp p,V\rp\in\RR^d\times\RR^d.
\end{equation}
With this definition of the running cost,  the first three lines of  (\ref{eq:MFGC}) can be written as follows:
    \begin{align}
    \label{eq:HJB}
        &-\ptt u
        - \nu\Delta u
        + \frac12 \labs \nabla_xu-\ll\th V\rabs^2
        -\frac{\ll^2\th}2\labs V\rabs^2
        = cm+f_0(x),\quad &\text{ in }
            [0,T)\times\Omega,
        \\
        \label{eq:FPK}
        & \ptt m
        - \nu\Delta m
        -\divo\lp \lp \nabla_xu
        -\ll\th V\rp m\rp
        =
        0,\quad &\text{ in }
            (0,T]\times\Omega,
        \\
        \label{eq:defV}
        &V(t,x)
        =
        -    \frac 1 {Z(t,x)}\int_{\Omega}
        \lp\nabla_xu(t,y)-\ll\th V(t,y)\rp K(x,y)
dm(t,y),  &\text{ in }
            (0,T)\times\Omega,
        \\
        \label{eq:defZ}
        &Z(t,x)
        =
        \int_{\Omega}
        K(x,y)dm(t,y),
 &\text{ in }
            [0,T)\times\Omega,
            \\
\label{eq:3}            &u(T,x)
            =\phi(x),
            &\text{ in }
            \Omega, 
\\ \label{eq:4}
            &m(0,x)
            =m_0(x),
            &\text{ in }
            \Omega.
\end{align}
We can now specify the boundary conditions on $(0,T)\times \partial \Omega$.
First, assuming that $u$ is smooth enough, the optimal control of a representative player is given by the feedback law:
\begin{displaymath}
  \alpha^* (t, x) = - \Bigl(\nabla_x u (t, x)-\lambda \theta V(t,x)\Bigr).
\end{displaymath}
The reflection condition at the boundary translates  into the Neumann boundary conditions:
\begin{equation}
      \label{eq:BCNeuu}
    \frac{\partial u}{\partial n}(t,x)=0,
    \qquad \text{ on } [0,T)\times\partial\Omega.
\end{equation}
Now, from the definition of $m$,  $\ds \int_\Omega m(t,x)\psi(x) dx = \EE\left(\int_\Omega m_0(x) \psi(X_{t,x}) dx\right)$ for all smooth enough test-function $\psi$  such that $ \frac{\partial \psi}{\partial n}(x)=0$ on $\partial \Omega$. Taking the time derivative of the latter equality, and using (\ref{eq:BCNeuu}), we deduce that 
\begin{equation}
  \label{eq:BCNeum}
   \nu \frac{\partial m}{\partial n}(t,x) -\lambda\theta  m (t,x) V(t,x)\cdot n(x)
    =
    0,
    \qquad \text{ on } (0,T]\times\partial\Omega.
  \end{equation}
  Note that (\ref{eq:FPK}), (\ref{eq:BCNeuu}) and (\ref{eq:BCNeum}) imply that the total mass $\int_\Omega m(t,x)dx$ is conserved.

We have studied or will sometimes consider other kinds of boundary conditions on some parts of the boundary, for example:
\begin{itemize}
\item periodic conditions, i.e. we set (\ref{eq:HJB})-(\ref{eq:4})) in the flat  torus $\Omega=\RR^d/\ZZ^d$
\item At a part of the boundary standing for an exit, there is an exit cost.  This yields a Dirichlet condition on $u$.
The Dirichlet condition on $m$: $m=0$ at exits, means that the agents stop taking part in the mean field game as soon as they reach the exit.
Such boundary  conditions will be simulated numerically in paragraph ~\ref{sec:second-example} below.
\item  If a part of the boundary stands for an entrance, then it is natural to impose a Dirichlet condition on $u$ to prevent the agents from  crossing the entrance outward, and a flux-like condition on $m$ to specify the entering flux of agents, see also  paragraph ~\ref{sec:second-example} below.
\end{itemize}

\subsection{Organization of the paper}
\label{sec:organization-paper}

Section \ref{sec:finite-diff-meth} is devoted to the description of the finite difference scheme; it is based on a monotone upwind scheme for the HJB  equation, and a scheme for the FP equation which is obtained by differentiating the discrete HJB equation and taking the adjoint. It is an extension of the finite difference schemes proposed and studied by the first author and I. Capuzzo-Dolcetta in \cite{MR2679575,MR3135339} for simpler MFGs.
Designing an efficient strategy for solving the system of non linear equations arising in the discrete version 
of MFGCs is a challenging task, in particular because 
\begin{itemize}
\item the underlying system of PDEs is forward-backward and cannot be solved by simply marching in time 
\item there is no underlying variational structure
\item Equation (\ref{eq:defV}) is non local.
\end{itemize}
 In Section~\ref{sec:newt-algor-solv}, we propose a strategy for solving the system of non linear equations:
it is a continuation method in which the viscosity parameter (for instance) is progressively decreased to the desired value, and for each value of the latter parameter,
the system of non linear equations is solved by a Newton method with inner iterations of a bi-gradient like algorithm.

Finally, in Section \ref{sec:numer-simul}, we discuss the results of numerical  simulations in two cases. In the first example, we show in particular that there exist multiple solutions and we discuss how the iteration count of the solver is affected by the variation of the parameters in the model. The second case is a model for a crowd of agents crossing a rectangular hall from an entrance to an exit:  we consider situations in which queues occur, and we show how the solution  depends on the parameters.

\section{Finite difference methods}\label{sec:finite-diff-meth}
In order to approximate (\ref{eq:HJB})-(\ref{eq:4})  numerically, we propose  a  finite difference  method reminiscent of that introduced in \cite{MR2679575, MR3135339} for MFGs without coupling through the controls. The important features of such methods are as follows
\begin{itemize}
\item they are based on monotone finite difference schemes for (\ref{eq:HJB}). Hence, a comparison principle still holds at the discrete level.
\item the special structure of  (\ref{eq:HJB})-(\ref{eq:FPK}) is preserved at the discrete level, namely that the FP equation can be obtained by differentiating the HJB equation w.r.t. $u$ and by taking the adjoint of the resulting equation. This results in a monotone approximation of (\ref{eq:FPK}), which ensures that the discrete version  of $m$ remains non negative at all time if $m_0 $ is non negative. The discrete FP equation will also preserve mass.
\end{itemize}
To simplify the discussion, let us focus on the case when $d=2$ and $\Omega=(0,1)^2$.
  More complex domains (even with holes) can be handled by the present method, but an additional effort would be needed if the domain had boundaries not aligned with the axes of the underlying grid that will be introduced soon. We also assume for simplicity that the boundary conditions 
are of the type (\ref{eq:BCNeuu})-(\ref{eq:BCNeum}) on the whole $\partial \Omega$.

\subsection{Notations and definitions}
\label{sec:notations}

For two positive integers $N_T, N_h$,
we set $\Delta t=T/N_T$, the time step,
and $h=1/N_h$, the  step related to the state variables. Consider the set of discrete times
$\cT_{\Delta t}=\lc t_k=k\Delta t, k=0,\dots,N_T\rc$
and the  grid $\Omega_{h}=\lc x_{i,j}=(ih,jh), i,j=0,\dots,N_h\rc$.

The goal will be to approximate $u\lp t_k,x_{i,j}\rp$ and $m\lp t_k,x_{i,j}\rp$ respectively by 
$u^k_{i,j}$ and $m^k_{i,j}$,  for all $k\in \{0,\dots, N_T\}$ and $(i,j)\in \{0,\dots, N_h\}^2$,
by solving the discrete version of (\ref{eq:HJB})-(\ref{eq:BCNeum}) arising from the finite difference scheme.\\

Let us define the discrete time derivative of $y:\cT_{\Delta t}\times\Omega_h\rightarrow\RR$ by the collection of real numbers:
\begin{equation}
    D_ty^{k}_{i,j}
    =
    \frac{y^{k+1}_{i,j}-y^k_{i,j}}{\Delta t},
\end{equation}
for $k=0,\dots, N_T-1$, $i,j=0,\dots,N_h$.

Given a grid function $y:\Omega_h\rightarrow\RR$,
we introduce the first order right sided  difference operators
\begin{equation}
    \label{eq:defFDO}
    \begin{aligned}
        \lp D_1^+y\rp_{i,j}
        &=
        \frac{y_{i+1,j}-y_{i,j}}{h},\\
        \lp D_2^+y\rp_{i,j}
        &=
        \frac{y_{i,j+1}-y_{i,j}}{h},
      \end{aligned}
    \end{equation}      
    and the five point  discrete laplace operator:
    \begin{equation}
   \label{eq:5}
\lp\Delta_h y\rp_{i,j}
       =
        -\frac{1}{h^2}
        \lp4y_{i,j}
        -y_{i+1,j}-y_{i-1,j}
        -y_{i,j+1}-y_{i,j-1}
        \rp,
\end{equation}
for all $(i,j)$ such that $0\le i,j\le N_h$. Note that the definition of these operators at boundary nodes of $\Omega_h$ needs to extend the grid function $y$ 
on a layer of nodes outside $\Omega_h$. This is  done by using discrete versions of the Neumann conditions (\ref{eq:BCNeuu})-~(\ref{eq:BCNeum}):
 assume that the boundary condition for $y$ is of  the type  $   \frac{\partial y}{\partial n} = z$, where $z$ is a given continuous function. 
Let  $(z^k_{i,j})$ is a suitable grid function approximating  $z$. We then choose the discrete version of the latter equation (first order scheme):
\begin{equation}
    \begin{aligned}
        y^k_{-1,j}
        &=
        y^k_{0,j}+hz^k_{0,j}
        \quad \text{ and } \quad
        y^k_{N_h+1,j}
        =
        y^k_{N_h,j}+hz^k_{N_h,j},
        \\
        y^k_{i,-1}
        &=
        y^k_{i,0}+hz^k_{i,0}
        \quad \text{ and } \quad
        y^k_{i,N_h+1}
        =
        y^k_{i,N_h}+hz^k_{i,N_h},
    \end{aligned}
\end{equation}
for $i,j=0,\dots,N_h$ and $k=0,\dots,N_T$.
\begin{remark}
  \label{sec:finite-diff-meth-2}
We have chosen a first order scheme for the boundary condition in order to preserve the monotonicity of the discrete Hamiltonian at boundary nodes, see later. Since the overall scheme is monotone  as it was already mentioned above, thus first order, it would be pointless to choose a higher order scheme for the boundary conditions.
\end{remark}
\begin{remark}\label{sec:finite-diff-meth-1}
Note that the previously mentioned discrete operators   are not needed at the nodes of $\partial \Omega_h$ at which Dirichlet boundary conditions are imposed.  
\end{remark}
For a grid function $V:\Omega_h\rightarrow\RR$,
we define the value of $V$ at half-integer indices by linear interpolation:
\begin{equation}
    \label{eq:interpV}
    \begin{aligned}
        V_{i+\frac12,j}
        &=
        \frac{V_{i+1,j}+V_{i,j}}2,\quad
        V_{-\frac12,j}
        =
        V_{0,j},
        \text{ and }
        V_{N_h+\frac12,j}
        =
        V_{N_h,j},
        \quad
        0\leq i\leq N_h-1,
        0\leq j\leq N_h,
        \\
        V_{i,j+\frac12}
        &=
        \frac{V_{i,j+1}+V_{i,j}}2, \quad 
        V_{i,-\frac12}
        =
        V_{i,0},
        \text{ and }
        V_{i,N_h+\frac12}
        =
        V_{i,N_h},
        \quad
        0\leq i\leq N_h,
        0\leq j\leq N_h-1.
    \end{aligned}
\end{equation}

In order to define the Godunov scheme for (\ref{eq:HJB}), we introduce the map $\Phi$: for two grid functions $y:\Omega_h\to \RR$ and $V: \Omega_h\to \RR^2$,
the grid function $\Phi(y,V): \Omega_h\to \RR^4$ is  defined by 
\begin{equation}
    \label{eq:defq}
   \Phi(y,V)_{i,j}
    = \left\{
      \begin{array}[c]{c}
\lb\lp D_1^+y\rp_{i,j}-\ll\th V_{i+\frac12,j,1}\rb_{-}\\
    -\lb\lp D_1^+y\rp_{i-1,j}-\ll\th V_{i-\frac12,j,1}\rb_{+}\\
    \lb\lp D_2^+y\rp_{i,j}-\ll\th V_{i,j+\frac12,2}\rb_{-}\\
    -\lb\lp D_2^+y\rp_{i,j-1}-\ll\th V_{i,j-\frac12,2}\rb_{+}
      \end{array}
         \right\} \quad \quad \text{for }i,j=0,\dots,N_h.
\end{equation}

Finally, let us introduce the function  $g: \RR^4\times\RR^2\to \RR$:
\begin{equation}\label{eq:7}
    g(q,V)
    =
    \frac12\labs q\rabs^2
    -\frac{\ll^2\th}2\labs V\rabs^2.
\end{equation}

\subsection{The scheme}
\label{sec:scheme}
With the ingredients defined in paragraph \ref{sec:notations}, we are  ready to propose the discrete version of (\ref{eq:HJB})-~(\ref{eq:defZ}):
\begin{align}
    \label{eq:disHJB}
    &-D_t u^k_{i,j}
    -\nu\lp\Delta_h u^k\rp_{i,j}
    +g\lp \Phi\lp u^k,V^k\rp_{i,j},V^k_{i,j}\rp
    =
    cm^{k+1}_{i,j}
    +f_0(x_{i,j}),
    \\
    \label{eq:disFPK}
    &D_tm^k_{i,j}
    -\nu\lp\Delta_hm^{k+1}\rp_{i,j}
    -\cT\lp \Phi\lp u^k,V^k\rp,m^{k+1}\rp_{i,j}
    =0,
    \\
    \label{eq:disdefV1}
    &V^k_{i,j,1}
    =-\frac 1 {Z^k_{i,j}}
    \sum_{r,s}
    \lp \frac{u^{k}_{r+1,s}-u^k_{r-1,s}}{2h}
    -\ll\th V^k_{r,s,1}\rp
    K\lp x_{i,j},x_{r,s}\rp
    m^{k+1}_{r,s},
    \\
    \label{eq:disdefV2}
    &V^k_{i,j,2}
    =-\frac 1 {Z^k_{i,j}}
    \sum_{r,s}
    \lp \frac{u^{k}_{r,s+1}-u^k_{r,s-1}}{2h}
    -\ll\th V^k_{r,s,2}\rp
    K\lp x_{i,j},x_{r,s}\rp
    m^{k+1}_{r,s},
    \\
    \label{eq:disdefZ}
    &Z^k_{i,j}
    =
    \sum_{r,s}
    K\lp x_{i,j},x_{r,s}\rp
    m^{k+1}_{r,s},
\end{align}
for $i,j=0,\dots,N_h$, $k=0,\dots,N_T-1$, where  the discrete transport operator
$\cT$ is  defined by
\begin{multline}
    \label{eq:defT}
    \cT\lp q,m\rp_{i,j}
    =
    \frac1h\lp
    -q_{i,j,1}m_{i,j}
    +q_{i-1,j,1}m_{i-1,j}
    +q_{i,j,2}m_{i,j}
    -q_{i+1,j,2}m_{i+1,j}
    \right.
    \\
    \left.
    -q_{i,j,3}m_{i,j}
    +q_{i,j-1,3}m_{i,j-1}
    +q_{i,j,4}m_{i,j}
    -q_{i,j+1,4}m_{i,j+1}
    \rp,
\end{multline}
for any $i,j=0,\dots,N_h$.  Note that \eqref{eq:defq} is not enough in order to fully determine $\cT(\Phi(u, V),m)$ at the boundary nodes. We also need the following quantities:
\begin{equation}
    \label{eq:BCq}
    \begin{aligned}
        \Phi(u,V)_{-1,j,1}
        &=
        \lb\lp D_1^+u\rp_{-1,j}-\ll\th V_{0,j,1}\rb_{-}    =   \lb\ll\th V_{0,j,1}\rb_{+} ,
        \\
        \Phi(u,V)_{N_h+1,j,2}
        &=
        -\lb\lp D_1^+u\rp_{N_h,j}-\ll\th V_{N_h,j,1}\rb_{+}= -\lb\ll\th V_{N_h,j,1}\rb_{-}   ,
        \\
        \Phi(u,V)_{i,-1,3}
        &=
        \lb\lp D_2^+u\rp_{i,-1}-\ll\th V_{i,0,2}\rb_{-}= \lb \ll\th V_{i,0,2}\rb_{+}  ,
        \\
        \Phi(u,V)_{i,N_h+1,4}
        &=
        -\lb\lp D_2^+u\rp_{i,N_h}-\ll\th V_{i,N_h,2}\rb_{+}= -\lb\ll\th V_{i,N_h,2}\rb_{-} ,
    \end{aligned}
\end{equation}
for $i,j=0,\dots,N_h$, where the last identity in each line comes from (\ref{eq:disBCNeuu}) below.
The discrete version of (\ref{eq:3}) and (\ref{eq:4}) is  
\begin{equation}
    \label{eq:ICTC}
    u^{N_T}_{i,j}
    =
    \phi(x_{i,j}),
    \text{ and }
    m^0_{i,j}
    =
    m_0\lp x_{i,j}\rp,
    \quad 0\leq i,j\leq N_h.
\end{equation}

The discrete version of  (\ref{eq:BCNeuu})-~(\ref{eq:BCNeum})  is 
\begin{align}
    \label{eq:disBCNeuu}
  &\!\lc\begin{aligned}
  &u^k_{-1,j}
        =
        u^k_{0,j}
  \quad \text{ and } \quad      
        u^k_{N+1,j} 
        =
        u^k_{N,j},
        \\
        &u^k_{i,-1}
        =
        u^k_{i,0}
         \quad \text{ and } \quad      
        u^k_{i,N+1}
        =
        u^k_{i,N},
    \end{aligned}
    \right.
    \\
    \label{eq:disBCNeum}
  &\!\lc\begin{aligned}
  &\nu\lp m^{k+1}_{-1,j}-m^{k+1}_{0,j}\rp
    &-h\lb\ll\th V^k_{0,j,1}\rb_+m^{k+1}_{0,j}
    &+h\lb\ll\th V^k_{0,j,1}\rb_-m^{k+1}_{-1,j}
    &=
    0,
    \\
    &\nu\lp m^{k+1}_{N,j}-m^{k+1}_{N+1,j}\rp
    &+h\lb\ll\th V^k_{N,j,1}\rb_+m^{k+1}_{N,j}
    &-h\lb\ll\th V^k_{N,j,1}\rb_-m^{k+1}_{N+1,j}
    &=
    0,
    \\
    &\nu\lp m^{k+1}_{i,-1}-m^{k+1}_{i,0}\rp
    &-h\lb\ll\th V^k_{i,0,1}\rb_+m^{k+1}_{i,0}
    &+h\lb\ll\th V^k_{i,0,1}\rb_-m^{k+1}_{i,-1}
    &=
    0,
    \\
    &\nu\lp m^{k+1}_{i,N}-m^{k+1}_{i,N+1}\rp
    &+h\lb\ll\th V^k_{i,N,1}\rb_+m^{k+1}_{i,N}
    &-h\lb\ll\th V^k_{i,N,1}\rb_-m^{k+1}_{i,N+1}
    &=
    0,
    \end{aligned}
    \right.
\end{align}
for $i,j=0,\dots,N_h$ and $k=0,\dots,N_T-1$.

Note that  (\ref{eq:disHJB}) is an implicit scheme for (\ref{eq:HJB}), (recall that (\ref{eq:HJB}) is backward w.r.t. time), whereas (\ref{eq:disFPK}) is an implicit scheme for (\ref{eq:FPK}),  (recall that (\ref{eq:FPK}) is forward w.r.t. time). This explains why no restriction is made on the time step.

The discrete Hamiltonian introduced in (\ref{eq:7}) $g: \RR^4\times\RR^2\to \RR$, $(q,V)\mapsto g(q,V)$,  has the following properties:
\begin{enumerate}[label=\bf{H\arabic*}]
    \item (monotonicity)
        \label{hypo:gmono}
        if $q=(q_1,q_2,q_3,q_4)$, then 
        $g$ is nonincreasing with respect to $q_1$
        and $q_3$ and nondecreasing with respect to
        $q_2$ and $q_4$
    \item (consistency)
        \label{hypo:gconsistent}
        $g\lp q_1,q_1,q_3,q_3,V\rp
        =H\lp q_1,q_3,V\rp$   
    \item (regularity)
        \label{hypo:gC1}
        $g$ is $C^1$.
\end{enumerate}

\begin{remark}
  \label{sec:scheme-1}
We aim at using  Newton iterations in order to  solve the system of nonlinear equations arising from the discrete scheme.
For that purpose, we need to  linearize the discrete version of the Fokker-Planck equation. Therefore,
for another small positive parameter $\ee$, we may replace the definition (\ref{eq:defq}) of $\Phi$ by the following:
\begin{equation}
    \label{eq:defq2}
   \Phi(y,V)_{i,j}
    = \left\{
      \begin{array}[c]{c}
\lb\lp D_1^+y\rp_{i,j}-\ll\th V_{i+\frac12,j,1}\rb_{-,\ee}\\
    -\lb\lp D_1^+y\rp_{i-1,j}-\ll\th V_{i-\frac12,j,1}\rb_{+,\ee}\\
    \lb\lp D_2^+y\rp_{i,j}-\ll\th V_{i,j+\frac12,2}\rb_{-,\ee}\\
    -\lb\lp D_2^+y\rp_{i,j-1}-\ll\th V_{i,j-\frac12,2}\rb_{+,\ee}
      \end{array}
         \right\} \quad \quad \text{for }i,j=0,\dots,N_h,
\end{equation}
 where the $C^1$ approximation  $a\mapsto a_{+,\ee}$
of $a\mapsto a_+= a \mathbf{1}_{a>0}$ is defined by
\begin{equation}\label{eq:8}
    a_{+,\ee}
    =
    \frac12\lp e^{-\frac{|a|}{\ee}} -1\rp\ee
    +\mathbf{1}_{a>0}a,
\end{equation}
and $a_{-,\ee} = a_{+,\ee}-a$. We may modify (\ref{eq:BCq}) and (\ref{eq:disBCNeum}) accordingly.
\end{remark}

\begin{lemma}\label{sec:scheme-2}
    Assume that $(u,m,V)$ satisfies
    \eqref{eq:disHJB}-\eqref{eq:disdefZ},
    \eqref{eq:ICTC}, and
    \eqref{eq:disBCNeuu}, \eqref{eq:disBCNeum}.

    The total mass of $m$ is preserved,
    i.e.
    \begin{equation}
        \sum_{1\leq i,j\leq N_h}
        m^k_{i,j}
        =
        \sum_{1\leq i,j\leq N_h}
        m^0_{i,j},
    \end{equation}
    for any $k=0,\cdots,N_T$.
\end{lemma}
\begin{proof}
      From (\ref{eq:5}),    we deduce that
      \begin{multline}
          \label{eq:9bis}
        \sum_{0\leq i,j\leq N_h}
          \nu\lp\Delta m^{k+1}\rp_{i,j}
          =      
          \frac{\nu}{h^2}
          \lb\sum_{j=0}^{N_h}\lp m^{k+1}_{N_h+1,j}
          - m^{k+1}_{N_h,j}
          - m^{k+1}_{-1,j}
          + m^{k+1}_{0,j}
          \rp
          \right.
          \\
           \left.
          +\sum_{i=0}^{N_h}\lp m^{k+1}_{i,N_h+1}
          - m^{k+1}_{i,N_h}
          - m^{k+1}_{i,-1}
          + m^{k+1}_{i,0}
          \rp\rb
          .
      \end{multline}
  Then we sum \eqref{eq:defT} for $i,j=0,\dots,N_h$:
    \begin{multline}\label{eq:9}
        \sum_{0\leq i,j\leq N_h}
        \cT(q,m)_{i,j}
        =
        \frac1h
        \sum_{j=1}^{N_h}
        -q_{N_h,j,1}m_{N_h,j}
        +q_{-1,j,1}m_{-1,j}
        +q_{0,j,2}m_{0,j}
        -q_{N_h+1,j,2}m_{N_h+1,j}
        \\
        +\frac1h
        \sum_{i=1}^{N_h}
        -q_{i,N_h,3}m_{i,N_h}
        +q_{i,-1,3}m_{i,-1}
        +q_{i,0,4}m_{i,0}
        -q_{i,N_h+1,4}m_{i,N_h+1}.
    \end{multline}
    From \eqref{eq:defq}, \eqref{eq:BCq}, \eqref{eq:disBCNeuu} and \eqref{eq:disBCNeum},
    the sums of the right hand sides of \eqref{eq:9} and \eqref{eq:9bis} vanish if $q=\Phi\lp u^k,V^k\rp$,   for any $k=0,\dots,N_T-1$.
    
From the observations above, summing  \eqref{eq:disFPK} on $i,j=0,\dots,N_h$  yields
\begin{displaymath}
        \sum_{0\leq i,j\leq N_h}
        m^{k+1}_{i,j}
        =
        \sum_{0\leq i,j\leq N_h}
        m^{k}_{i,j},
\end{displaymath}
then the desired result.
\end{proof}

\subsection{Solving the discrete version of the Hamilton-Jacobi-Bellman equation}\label{sec:solv-discr-vers}
For brevity, we will use the notation $(y_{i,j}^k)$ for a grid function, omitting that the indices $i,j$ take their values in $\{0,\dots, N_h\}$ and that $k$ takes its values in $\{0,\dots, N_T\}$. This paragraph is devoted to solving the system of nonlinear equations satisfied by the grid function $(u_{i,j}^k)$, 
given the grid functions $(m^k_{i,j})$ and  $(V_{i,j}^k)$. Let $\ft=(\ft^k_{i,j})$ 
denote the grid function defined by  
\begin{equation}
\label{eq:13}  
\ft_{i,j}^k  =   cm^{k+1}_{i,j} +f_0\lp x_{i,j}\rp.
\end{equation}
 We may then rewrite (\ref{eq:disHJB}),(\ref{eq:disBCNeuu}) and the first identity in (\ref{eq:ICTC}) in the  compact form 
 \begin{equation}
   \label{eq:10}
u=\cF_u\lp \ft,V\rp.
 \end{equation}
Finding $u$ given $m$ and $V$ amounts to solving a discrete version of a nonlinear parabolic equation posed backward in time
with Neumann boundary conditions. This is much simpler than solving the complete forward-backward system for $(u,m, V)$, because a backward time-marching procedure can be used. Since, as it was already observed  above, the scheme is implicit, each time step consists of solving the discrete version of a nonlinear elliptic partial differential equation which is local because $V$ is given. Starting from the terminal time step $N_T$ for which \eqref{eq:ICTC} gives an explicit formula for $u^{N_T}$, the  $k$-th step  of the backward loop consists of computing $u^k$ by solving (\ref{eq:disHJB}) and (\ref{eq:disBCNeuu}) given $u^{k+1}$, $m^{k+1}$ and $V^k$.  This is done by  means of  Newton iterations. 

For completeness, let us give a few details on Newton iterations. We introduce the operator $\cR_u$,
\begin{equation}
    \label{eq:cRu}
    \cR_u(u',v',f',V')_{i,j}
    =
    u'_{i,j}
    -v'_{i,j}
    +\Delta t\lb
    -\nu\lp\Delta_hu'\rp_{i,j}
    +g\lp\Phi\lp u',V'\rp_{i,j},V'_{i,j}\rp
    -f'_{i,j}\rb,
\end{equation}
for $u',v',f':\Omega_h\rightarrow\RR$, $V':\Omega_h\rightarrow\RR^2$,
and $i,j=0,\dots,N_h$.  Note that in (\ref{eq:cRu}),  $u'$ is extended outside $\Omega_h$ thanks to (\ref{eq:disBCNeuu}).
We aim at approximating the solution of 
\begin{equation}
  \label{eq:11}
 \cR_u(u',u^{k+1},\ft^k,V^k)=0.
\end{equation}
Starting  from an initial guess noted $u^{k,0}:\Omega_h\rightarrow\RR$,  the Newton iterations consist of
computing by  induction a sequence  $u^{k,\ell}$ of approximations of the solution to (\ref{eq:11}): 
given $u^{k,\ell}$, $u^{k,\ell+1}$ is found  by solving the system of linear equations
\begin{displaymath}
    d_u\cR_u\lp u^{k,\ell}, u^{k+1},\ft^k,V^k\rp  (u^{k,\ell+1}-
    u^{k,\ell})
=
    -\cR_u\lp u^{k,\ell},u^{k+1},\ft^k,V^k\rp  .
\end{displaymath}
In the latter equation, the Jacobian matrix
$ d_u\cR_u\lp u^{k,\ell}, u^{k+1},\ft^k,V^k\rp $
is sparse since the PDE is local,
and invertible since the scheme is monotone.  Note that $ d_u\cR_u\lp u^{k,\ell}, u^{k+1},\ft^k,V^k\rp $  depends neither on  $ u^{k+1}$ nor on  $\ft^k$, so we can write it   $ d_u\cR_u\lp u^{k,\ell}, \cdot,\cdot,V^k\rp $. The system of linear equations is solved by using special algorithms for
sparse matrices. In the numerical simulations presented below, we use the C-library UMFPACK, see \cite{umfpack},
implementing  unsymmetric multifrontal method and direct sparse LU factorization. 

Note that one may choose the initial guess $u^{k,0}=u^{k+1}$.
The Newton iterations are stopped when the residual $| \cR_u(u^{k,\ell},u^{k+1},\tilde f^k,V^k)|$ is small enough, say for $\ell=\ell_0$.
Then we set $u^{k}=u^{k,\ell_0}$. It is well known that the Newton algorithm for (\ref{eq:disHJB}) is equivalent to an optimal policy iteration algorithm,
 that it is convergent for any initial guess, and that the convergence is quadratic. 

\subsection{Solving the discrete version of the Fokker-Planck-Kolmogorov equation}\label{sec:solv-discr-vers-1}

This paragraph is devoted to solving the system of linear equations satisfied by the grid function $(m_{i,j}^k)$, 
given the grid functions $(u^k_{i,j})$ and  $(V_{i,j}^k)$.
 We may then rewrite (\ref{eq:disFPK}),~(\ref{eq:disBCNeum}) and the second identity in (\ref{eq:ICTC}) in the  compact form 
 \begin{equation}
\label{eq:12}
m=\cF_m\lp u,V\rp.
 \end{equation}
Finding $m$ given $u$ and $V$ amounts to solving a discrete version of an  linear parabolic equation posed forward in time
with Neumann boundary conditions. Since the scheme is implicit, each time step consists of solving the discrete version of a 
linear elliptic partial differential equation which is  local. Starting from the terminal time step $0$ for which \eqref{eq:ICTC} gives an explicit formula for $m^{0}$, the  $k$-th step  of the forward loop consists of computing $m^{k+1}$ by solving  (\ref{eq:disFPK}), (\ref{eq:disBCNeum})  given $m^{k}$, $u^{k+1}$ and $V^k$.  It is easy to see that the matrix of the latter system of linear equations is exactly
the transposed of  $d_u\cR_u\lp u^{k}, \cdot,\cdot,V^k\rp$ which has been introduced in the previous paragraph. Therefore, it is sparse and invertible.  In the numerical simulations presented below, we use  the C-library UMFPACK again.

\section{Newton algorithms for solving the whole system (\ref{eq:disHJB})-(\ref{eq:disBCNeum})}
\label{sec:newt-algor-solv}
The solution of the system of nonlinear equations (\ref{eq:disHJB})-(\ref{eq:disBCNeum} is not easy because 
\begin{enumerate}
\item the system is the discrete version of a system of forward-backward equations, which precludes a simple time marching algorithm
\item The easiest instances of MFGs can be interpreted as optimal control problems driven by a partial differential equation, which opens the way 
to algorithms based on the variational structure. In the MFGC  considered here, there is no underlying variational principle and these algorithms cannot be used.
\end{enumerate}
Following previous works of the first author, see \cite{MR3135339}, we choose to use a continuation method (for example with respect to the viscosity parameter $\nu$) in which every system of nonlinear equations (given the parameter of the continuation method) is solved by means of Newton iterations.
With  Newton algorithm, it is important to have a good initial guess of the solution; for that,  we take advantage of the continuation method by  choosing the initial guess as the solution obtained with the previous value of  the parameter. Alternatively, we have sometimes taken the initial guess from the simulation of the same problem on a coarser grid, using interpolation. 

It is also important  to implement the Newton algorithm on a ``well conditioned''  system. Therefore,  we shall not directly address (\ref{eq:disHJB})-(\ref{eq:disBCNeum}), but we shall rather eliminate the unknowns $u$ and $m$ by using the time-marching loops  described in paragraphs \ref{sec:solv-discr-vers} and \ref{sec:solv-discr-vers-1} and see (\ref{eq:disHJB})-(\ref{eq:disBCNeum}) as a fixed point problem for $(\tilde f,V)$. 

Before describing the algorithm, we need to provide a numerical method for obtaining the average drift given $u$ and $m$,
i.e. for solving (\ref{eq:disdefV1})-(\ref{eq:disdefV2}) at least approximately.

\subsection{The coupling cost and the average drift}\label{sec:coupl-terms-aver}
Let us introduce  $\cF_f$ which maps the grid function $m$ defined on $\cT_{\Delta t}\times \Omega_h$ to the grid function $\ft$ 
given in (\ref{eq:13}).
We also define the maps $\cZ$ and $\cV$ by
\begin{align}
    \label{eq:defcZ}
    \cZ(m')_{i,j}
    &=
    \sum_{r,s}
    K\lp x_{i,j},x_{r,s}\rp
    m'_{r,s},
    \\
    \label{eq:defcV1}
    \cV(u',m',V')_{i,j,1}
    &=
    -\sum_{r,s}
    \lp \frac{u'_{r+1,s}-u'_{r-1,s}}{2h}
    -\ll\th V'_{r,s,1}\rp
    K\lp x_{i,j},x_{r,s}\rp
    \frac{m'_{r,s}}{\cZ(m')_{i,j}},
    \\
    \label{eq:defcV2}
    \cV(u',m',V')_{i,j,2}
    &=
    -\sum_{r,s}
    \lp \frac{u'_{r,s+1}-u'_{r,s-1}}{2h}
    -\ll\th V'_{r,s,2}\rp
    K\lp x_{i,j},x_{r,s}\rp
    \frac{m'_{r,s}}{\cZ(m')_{i,j}},
\end{align}
for $u',m':\Omega_h\rightarrow\RR$,
$V':\Omega_h\rightarrow\RR^2$
and $i,j=0,\dots,N_h$.

It can be proved exactly in the same manner as in the continuous case, see  \cite[Lemma 2.4]{2019arXiv190411292K}, that 
$V'\mapsto \cV(u',m',V')$ is a contraction in the maximum-norm for instance, with a contraction factor $|\ll|\th$.

For a positive integer $L$, we define the map  $\cF_V$   by :
\begin{eqnarray}
  \label{eq:14}
\cF_V\lp u,m,V\rp&=&\Vt,\\
 \Vt^k&=& \Vt^{k,L}, \qquad \text{for }k=0,\dots, N_T-1,
\end{eqnarray}
where $ \Vt^{k,0}
    =
    V^k$ and the sequence $\Vt^{k,\ell}$ is defined by the following induction:
\begin{equation}
    \label{eq:defVt}
    \Vt^{k,\ell}
    =
    \cV\lp u^k,m^{k+1},\Vt^{k,\ell-1}\rp,\qquad 
    1\leq \ell \leq L.
  \end{equation}
  
\begin{remark}\label{sec:coupl-terms-aver-1}
    \begin{enumerate}
        \item
          With a slight abuse of notation, if  $L=\infty$, then $\Vt^k$ is the  fixed
          point of the map $V'\mapsto \cV\lp u^k,m^{k+1},V' \rp$.
        \item
          The induction (\ref{eq:defVt}) corresponds to Jacobi iterations and can be easily parallelised. It is also possible to  implement Gauss-Seidel iterations, which are a little more complex to write. They consist of using the  components of $\Vt^{k, \ell}$ as soon as they are obtained (instead of those of $\Vt^{k, \ell-1}$) in order to compute the  components of  $\Vt^{k, \ell}$ which have not been obtained yet . In our implementation, we have in fact used the Gauss-Seidel iterations with a lexicographic ordering of the components of the grid functions.
    \end{enumerate}
\end{remark}
\begin{remark}\label{sec:coupl-cost-aver}
  As we shall see in paragraph \ref{sec:behaviour-algorithm}, the choice of $L$ does not impact the convergence of the overall iterative algorithm. Therefore, a good choice turns out to be  $L=1$.
\end{remark}
\subsection{The linearized operators}\label{sec:linearized-operators}
\subsubsection{Notation}
\label{sec:notation}
Let $u,m,f:\cT_{\Delta h}\times \Omega_h \to\RR$ be generic grid functions 
standing respectively for discrete versions of the value function, the law of the distribution of states, the right hand side of the discrete HJB equation (\ref{eq:disHJB}).   Let  $V:\cT_{\Delta h}\times \Omega_h\to \RR^2$ be a generic grid function  standing  for the average drift.
Let us introduce the operators obtained by differentiation of the maps $\cF_u$, $\cF_m$ and $\cF_V$:
\begin{equation}\label{eq:26}
    \begin{aligned}
        B_{u,f}
        &=
        D_f\cF_u(f,V),
        \quad
        &B_{u,V}
        &=
        D_V\cF_u(f,V),
        \\
        B_{m,V}
        &=
        D_u\cF_m(u,V),
        \quad
        &B_{m,V}
        &=
        D_V\cF_m(u,V),
        \\
        C_{f,m}
        &=
        D_m\cF_f(m),
        \quad
        &C_{V,u}
        &=
        D_u\cF_{V}(u,m,V),
        \\
        C_{V,m}
        &=
        D_m\cF_{V}(u,m,V),
        \quad
        &C_{V,V}
        &=
        D_V\cF_{V}(u,m,V).
    \end{aligned}
\end{equation}
In the three paragraphs below, we explain how these differential operators can be computed.


\subsubsection{Linearized Hamilton-Jacobi-Bellman equation}\label{sec:line-hamilt-jacobi}
The variation  $du$ of $u= \cF_u (f,V)$  induced by variations $df$ and $dV$ of $f$ and $V$ is given by 
\begin{equation}
\label{eq:15}
    du
    =
    \begin{pmatrix}
        B_{u,f}
        &B_{u,V}
    \end{pmatrix}
    \begin{pmatrix}
        df
        \\
        dV
    \end{pmatrix}.
\end{equation}
More explicitly, $du$ is obtained by linearizing (\ref{eq:disHJB}), (\ref{eq:disBCNeuu}) and (\ref{eq:ICTC}). It satisfies 
\begin{equation}
    \label{eq:linHJB}
    -D_t du^k_{i,j}
    -\nu\lp\Delta_h du^k\rp_{i,j}
    +\Phi\lp u^k,V^k\rp_{i,j}
    \cdot 
\Bigl(    d\Phi\lp u^k,V^k\rp
    \lp du^k,dV^k\rp\Bigr)_{i,j}
    -\ll^2\th V^k_{i,j}\cdot dV^k_{i,j}
    =
    df^k_{i,j},
\end{equation}
for $i,j=0,\dots,N_h$ and $k=0,\dots,N_T$.
The first (respectively second) inner product appearing in (\ref{eq:linHJB}) takes  place in $\RR^4$ (respectively $\RR^2$). The set of equations (\ref{eq:linHJB}) is supplemented with the terminal condition
\begin{equation}
    \label{eq:TCdu}
    du_{i,j}^{N_T}    =    0,
\end{equation}
and the boundary conditions 
\begin{equation}
    \label{eq:BCdu}
    \lc
    \begin{aligned}
        du^k_{-1,j}
        &=
        du^k_{0,j}
        ,
        \quad
        &du^k_{N_h+1,j}
        &=
        du^k_{N_h,j}
       ,
        \\
        du^k_{i,-1}
        &=
        du^k_{i,0}
        ,
        \quad
        &du^k_{i,N_h+1}
        &=
        du^k_{i,N_h}
       .
    \end{aligned}
    \right.
\end{equation}
Given $df$ and $dV$, the variation $du$ is found by solving  (\ref{eq:linHJB}),(\ref{eq:TCdu}) and (\ref{eq:BCdu}). This is done by marching backward in time and at each time step solving a system of linear equations with a  sparse   and invertible matrix (of the same form as in the Newton iterations described in \ref{sec:solv-discr-vers}). Again, we use the library UMFPACK for that.

\subsubsection{Linearized Kolmogorov-Fokker-Planck equation}\label{sec:line-kolm-fokk}
The variation  $dm$ of $m= \cF_m (u,V)$  induced by variations $du$ and $dV$ of $u$ and $V$ is given by 
\begin{equation}\label{eq:16}
    dm
    =
    \begin{pmatrix}
        B_{m,u}
        &B_{m,V}
    \end{pmatrix}
    \begin{pmatrix}
        du
        \\
        dV
    \end{pmatrix}.
\end{equation}
The grid function  $dm$ is obtained by linearizing (\ref{eq:disFPK}), (\ref{eq:disBCNeum}) and (\ref{eq:ICTC}). It satisfies 
\begin{equation}
    \label{eq:FPKdm}
    D_tdm^k_{i,j}
    -\nu\lp\Delta_hdm^{k+1}\rp_{i,j}
    -\cT_{i,j}\lp
d\Phi\lp u^k,V^k\rp
    \lp du^k,dV^k\rp\Bigr),m^{k+1}\rp
    -\cT_{i,j}\lp \Phi\lp u^k,V^k\rp,dm^{k+1}\rp
    =0,
\end{equation}
for $i,j=0,\dots,N_h$ and $k=0,\dots,N_T$. The set of equations (\ref{eq:FPKdm}) is supplemented with the initial condition
\begin{equation}\label{eq:18}
    dm^0_{i,j}
    =
    0,
    \quad
    i,j=0,\dots,N_h,
\end{equation}
and the boundary conditions

\begin{equation}\label{eq:17}
\lc \begin{aligned}
  &\nu\lp dm^{k+1}_{-1,j}-dm^{k+1}_{0,j}\rp
    &-h\lb\ll\th V^k_{0,j,1}\rb_+dm^{k+1}_{0,j}
    &+h\lb\ll\th V^k_{0,j,1}\rb_-dm^{k+1}_{-1,j}
    &=
    0,
    \\
    &\nu\lp dm^{k+1}_{N,j}-dm^{k+1}_{N+1,j}\rp
    &+h\lb\ll\th V^k_{N,j,1}\rb_+dm^{k+1}_{N,j}
    &-h\lb\ll\th V^k_{N,j,1}\rb_-dm^{k+1}_{N+1,j}
    &=
    0,
    \\
    &\nu\lp dm^{k+1}_{i,-1}-dm^{k+1}_{i,0}\rp
    &-h\lb\ll\th V^k_{i,0,1}\rb_+dm^{k+1}_{i,0}
    &+h\lb\ll\th V^k_{i,0,1}\rb_-dm^{k+1}_{i,-1}
    &=
    0,
    \\
    &\nu\lp dm^{k+1}_{i,N}-dm^{k+1}_{i,N+1}\rp
    &+h\lb\ll\th V^k_{i,N,1}\rb_+dm^{k+1}_{i,N}
    &-h\lb\ll\th V^k_{i,N,1}\rb_-dm^{k+1}_{i,N+1}
    &=
    0.
    \end{aligned}
    \right.
\end{equation}
Here again, (\ref{eq:FPKdm}), (\ref{eq:18}) and (\ref{eq:17}) is solved by marching in time and solving a system of linear equations at each time step
(using UMFPACK).

\subsubsection{Linearized coupling costs and average drifts}\label{sec:line-coupl-costs}
The variation of $\ft=\cF_f(m)$ induced by a variation $dm$ of $m$ is obviously given by $   d\ft_{i,j}^k    =    cdm^k_{i,j}$, hence $C_{f,m}dm=cdm$.
\\
Let us turn to the variation of $\Vt=\cF_V(u,m,V)$ induced by  the variations $du$, $dm$ and $dV$ of $u$, $m$ and $V$. Differentiating (\ref{eq:14})-(\ref{eq:defVt}) leads to 
\begin{align*}
    d\Vt^{k,0}
    &=
    dV^{k},
    \\
    d\Vt^{k,\ell}
    &=
    \cV\lp du^k,m^{k+1},d\Vt^{k,\ell-1}\rp
    +d_m\cV\lp u^k,m^{k+1},\Vt^{k,\ell-1}\rp dm^{k+1},
    \\
    d\Vt ^k 
    &=
    d\Vt^{k,L},
\end{align*}
see (\ref{eq:defVt}) for the definition of $\Vt^{k,L}$.

This permits to compute 
\begin{displaymath}
d\Vt= \Bigl(C_{V,u} \; C_{V,m} \; C_{V,V}\Bigr)   \begin{pmatrix}
        du
        \\
        dm
        \\
        dV
    \end{pmatrix}.
\end{displaymath}
To sumarize,
\begin{align*}
    \begin{pmatrix}
        d\ft
        \\
        d\Vt
    \end{pmatrix}
    &=
    C
    \begin{pmatrix}
        du
        \\
        dm
        \\
        dV
    \end{pmatrix},\qquad \text{with }\qquad C=  \begin{pmatrix} 0  & C_{fm} & 0\\ C_{V,u}& C_{V,m}& C_{V,V}
\end{pmatrix}.
\end{align*}

\subsection{The algorithm for solving (\ref{eq:disHJB})-(\ref{eq:disBCNeum})}
\label{sec:algor-solv-refeq:d}

We see (\ref{eq:disHJB})-(\ref{eq:disBCNeum}) as a fixed point problem for the pair $(f, V)$, which we write 
\begin{equation}\label{eq:19}
    \cG_f(f,V)=0,
    \quad
    \text{ and }
\quad     \cG_V(f,V)=0,
\end{equation}
where $\cG_f,\cG_V$ are defined by
\begin{eqnarray}
  \label{eq:20}
    \cG_f(f,V)
    &=&
    f-\cF_f\circ\cF_m\Bigl(\cF_u(f,V),V\Bigr),
    \\ \label{eq:21}
    \cG_V(f,V)
    &=&
    V-\cF_{V}\Bigl(\cF_u(f,V),\cF_m\lp\cF_u(f,V),V\rp,V\Bigr).
\end{eqnarray}

\begin{remark}
  Given $u$ and $m$,  $V\mapsto W=\cF_V\lp u,m,V\rp $ is obtained  as follows: for each $k$,  computing $W^k$  consists of 
  iterating  $W^k\times \cV^L(u^k,m^{k+1}, W^k)$  $L$ times
starting from $W^k=V^k$. Recall that  $\cV(u^k,m^{k+1},\cdot)$ 
is in \eqref{eq:defcV1} and \eqref{eq:defcV2}, and is a contraction with a unique fixed point.
Therefore, $\cF_V\lp u,m,\cdot\rp$ has also a unique fixed point  which does not depend on $L$. Since  $\cF_u$ and $\cF_m$ do not depend on $L$, neither do the  the solutions of \eqref{eq:19}.
\end{remark}

The Newton iterations involve the Jacobian $A(f, V)$ of the map $(f,V)\mapsto  \Bigl(\cG_f(f,V)), \cG_V(f,V)\Bigr)$. We set 
\begin{equation}\left\{
    \begin{aligned}
        &A_{f,f}(f,V)
        =
        D_f\cG_f(f,V),
        \quad
        &A_{f,V}(f,V)
        =
        D_V\cG_f(f,V),
        \\
        &A_{V,f}(f,V)
        =
        D_f\cG_V(f,V),
        \quad
        &A_{V,V}(f,V)
        =
        D_V\cG_V(f,V),
    \end{aligned}\right.
\end{equation}
and the blocks of the Jacobian $A(f,V)$ are defined by:
\begin{eqnarray}
  \label{eq:22}
        A_{f,f}(f,V)
        &=&
        I_N-C_{f,m}B_{m,u}B_{u,f},
        \\ \label{eq:23}
        A_{f,V}(f,V)
        &=&
        -C_{f,m}\lp B_{m,u}B_{u,V} + B_{m,V}\rp,
        \\\label{eq:24}
        A_{V,f}(f,V)
        &=&
        -C_{V,u}B_{u,f}
        -C_{V,m}B_{m,u}B_{u,f},
        \\\label{eq:25}
        A_{V,V}(f,V)
        &=&
        I_{2N}-C_{V,u}B_{u,V}-C_{V,m}\lp B_{m,u}B_{u,V} + B_{m,V}\rp
        -C_{V,V},
\end{eqnarray}
where $B_{u,f}$,  $B_{u,V}$, $B_{m,u}$,  $B_{m,V}$,   $C_{V,u}$, $C_{V,m}$,  $C_{V,V}$ are given by 
(\ref{eq:26}) with  $u=\cF_u(f,V)$ and   $m=\cF_m(u,V)$. In an equivalent manner, we can write 
\begin{equation}\label{eq:27}
    A(f,V)=
    I_{3N}
    -
    \begin{pmatrix}
        C_{f,m}B_{m,u}B_{u,f}
        &C_{f,m}\lp B_{m,u}B_{u,V} + B_{m,V}\rp\\
        C_{V,u}B_{u,f}
        +C_{V,m}B_{m,u}B_{u,f}
        &C_{V,u}B_{u,V}
        +C_{V,m}\lp B_{m,u}B_{u,V} + B_{m,V}\rp
        +C_{V,V}
    \end{pmatrix},
\end{equation}
or 
\begin{equation}\label{eq:28}
    A(f,V)
    =
    I_{3N}-
    \begin{pmatrix}
        0_N
        &C_{f,m}
        &0_{N,2N}
        \\
        C_{V,u}
        &C_{V,m}
        &C_{V,V}
    \end{pmatrix}
    \begin{pmatrix}
        I_N
        &0_{N,2N}
        \\
        B_{m,u}
        &B_{m,V}
        \\
        0_{2N,N}
        &I_{2N}
    \end{pmatrix}
    \begin{pmatrix}
        B_{u,f}
        &B_{u,V}
        \\
        0_{2N,N}
        &I_{2N}
    \end{pmatrix}.
\end{equation}
Every  Newton iteration for solving (\ref{eq:19}) consists of solving a system of linear equations of the form 
 \begin{equation}
    A(f^{\ell},V^{\ell})
    \begin{pmatrix}
        f^{\ell+1}- f^{\ell} \\
        V^{\ell+1}- V^{\ell} \\
    \end{pmatrix}
    =-
    \begin{pmatrix}
        \cG_f (f^{\ell},V^{\ell})\\
        \cG_V(f^{\ell},V^{\ell})
    \end{pmatrix}.
\end{equation}
In our implementation, this system is solved iteratively by BiCGStab algorithm, see \cite{MR1149111}.  Note that BiCGStab algorithm  only requires 
a function that computes
\begin{displaymath}
A(f^{\ell},V^{\ell})
\begin{pmatrix}
  \hat f\\\hat V
\end{pmatrix},
\end{displaymath}
for any grid functions $\hat f: \cT_{\Delta t}\times \Omega_h \to \RR$ and $\hat V: \cT_{\Delta t}\times \Omega_h \to \RR^2$. Ths is done using 
(\ref{eq:28}) and combining the methods described in paragraph \ref{sec:linearized-operators}.  The assembly of  the Jacobian matrix is not needed.

\section{Numerical simulations}
\label{sec:numer-simul}

We are going to discuss the results of numerical simulations in two cases, both  related to crowd motion.

\subsection{First example}
\label{sec:first-example}

\subsubsection{Description of the model}
\label{sec:description-model}

The state space is the square $\Omega=(-0.5,0.5)^2$ and the time horizon is $T=1$.

We consider the MFGC described by (\ref{eq:HJB})-(\ref{eq:BCNeum}), in which 
\begin{enumerate}
\item The kernel $k$ in (\ref{eq:defV}) is radial, i.e. of the form $K(x,y)=K_\rho(|x-y|)$. Here $r\mapsto K_\rho(r)$ is a non-increasing $C^1$ function defined on $\RR_+$, with       
  \begin{equation} \label{eq:29}
            K_\rho(r)
            =
            0,\quad  \text{ if } r\geq \rho,\qquad 
            \text{ and }\qquad 
            K_\rho(r)
            =
            1, \quad  \text{ if } r\leq 0.9\rho,
        \end{equation}
 for a positive number  $\rho>0$, which is the radius of the disc in the state space that a reprensative agent uses for computing the average of the controls.
\item The parameter $c$ in  (\ref{eq:HJB}) is chosen as  $c=10^{-3}$. Recall that the cost $cm(t,x)$ reflects the aversion of a representative player to crowded regions of the state space.
\item  At time $t=0$, the agents are distributed  in the top-right and the bottom-left corners of the domain $\Omega$, see the left part of Figure \ref{fig:CI_FC}. The density is piecewise constant, with values appearing on  Figure \ref{fig:CI_FC}.
\item The terminal cost is also piecewise constant. It takes a small value in  the top-left and the bottom-right corners of the domain $\Omega$, see the right part of Figure \ref{fig:CI_FC} for the chosen values. This cost  attracts the agents to the latter two corners of $\Omega$.
\item The function  $f_0$ in  (\ref{eq:HJB}) is chosen to be proportional to the terminal cost, with a factor $0.1$. This term is linked to the running cost and 
has the same effect as the terminal cost.
\item Recall that the boundary conditions (\ref{eq:BCNeuu})-(\ref{eq:disBCNeum}) rule out the entry or exit of agents. The total mass of $m$ is conserved.
\end{enumerate}
\begin{remark}
\label{sec:description-model-1}  
Note that the problem is symmetric with respect to the two diagonals of the square domain. The grid of the domain is chosen to have the same symmetry.
\end{remark}
\begin{remark}
  \label{sec:description-model-2}
The two parts  of the running cost:  $cm(t,X_t)$ and  $\frac{\theta}2\labs\aa-\ll V(t, X_t)\rabs^2$ 
may have competing effects: indeed, the former cost may incitate the agents to spread all over the state space, whereas 
the latter may result in the agents selecting the same controls therefore staying grouped. 
Although the constant $c= 10^{-3}$ seems small, it is chosen in such a way that the above-mentioned two costs have the same orders of magnitude  for e.g. $\theta=1$ and $\lambda=0.9$.
\end{remark}
\begin{center}
 \includegraphics[scale=0.5]{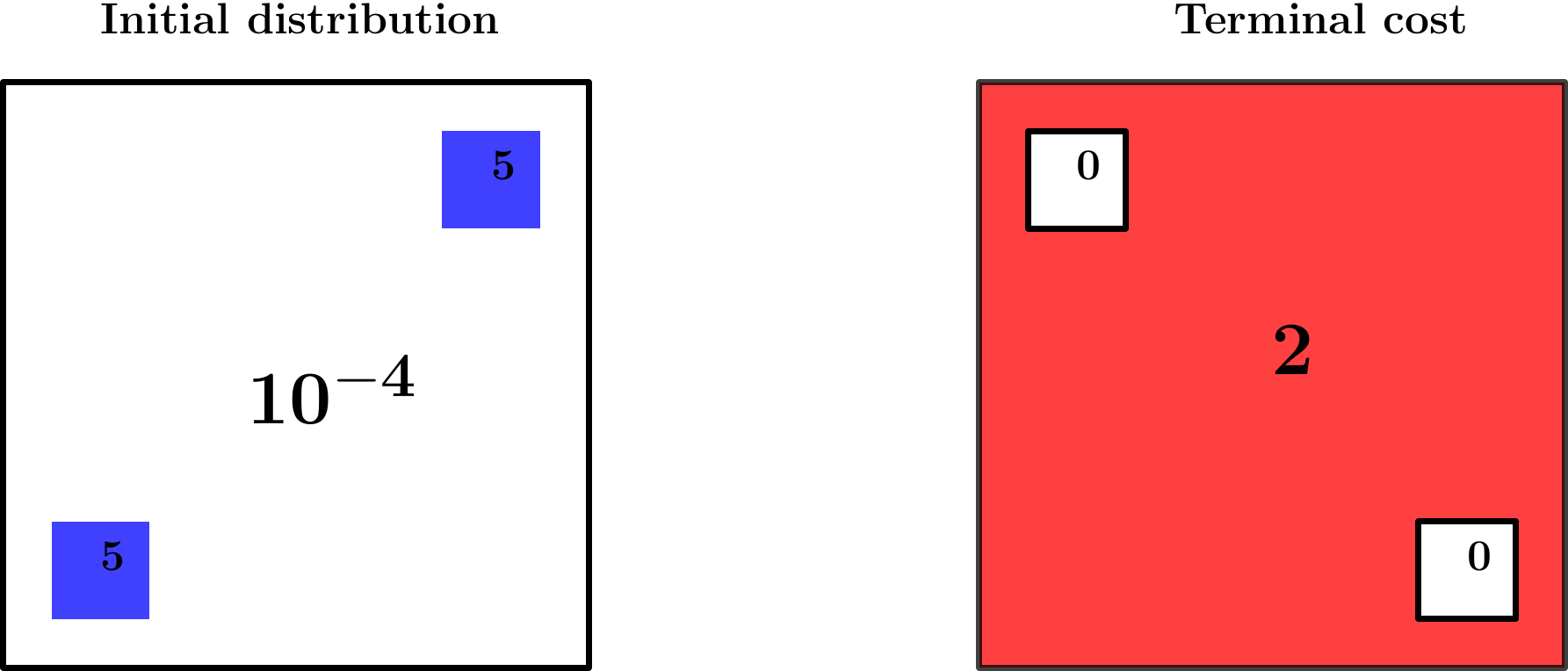}
  \captionof{figure}{Example 1. Left: the initial distribution of states. Right: the terminal cost.}
 \label{fig:CI_FC}
\end{center}

\subsubsection{The results of the simulations: discussion}
\label{sec:results-simulations}

In the present  simulation, we choose the parameters as follows:
\begin{displaymath}
 \nu=10^{-3},\qquad \theta=1,\qquad\lambda=0.9,\qquad \rho=0.2 ,\qquad c=10^{-3}. 
\end{displaymath}
On Figure \ref{fig:sym_cor_2D}, we display snapshots of  $m$, $u$ and the optimal feedback law
at several times.\\
Since both the problem and the grid are symmetric with respect to the two diagonals, $m$, $u$ and the feedback have the same symmetry for all times.
Let us describe the evolution of the distribution of states, that we can name ``{\sl gathering-kissing-splitting}'' 
referring to the ``{\sl drafing-kissing-tumbling}'' phenomenon in fluid mechanics (for the interaction of particles in a fluid).
\begin{itemize}
\item The term $cm(t, X_t)$ in the running cost prevents the part of the distribution initially supported  in one of the opposite corners (say the bottom-left corner) to travel directly to one of the targets (say the bottom-right corner). On the other hand, the interaction through controls prevents this part of the distribution to split into two equal parts which would travel directly to the two targets, because  the agents favor controls close to the local average. Therefore, 
the part of the distribution initially supported in one of the opposite corners first travels to the center of the domain. This is the {\sl gathering} phase of the evolution.
\item The two parts of the distribution reach the center of the domain, where  the local average of the velocity becomes  small;
the dominating cost then becomes the one which attracts the agents to the bottom-right and top-left corners. Because of the repulsive effect due to the part $cm(t, X_t)$ of the running cost, the part of the distribution initially supported  in one of the opposite corners (say the bottom-left corner) and having traveled toward the center of the domains splits into two parts which make each a ninety-degrees turn. This is the {\sl kissing} phase of the evolution.
\item After the {\sl kissing } phase, the distribution of states {\sl splits} into  parts which travel to the targets. 
\end{itemize}
Finally, we see that the paths followed by the agents is far from being a shortest path to the targets.
\bigskip

Let us compare these results with a simulation of the MFG  obtained by cancelling 
either $\theta$ or $\lambda$ in (\ref{eq:HJB})-(\ref{eq:BCNeum}) while keeping all the other parameters and the grid unchanged.
Since the problems remains symmetric with respect to the two diagonals, the obtained results have the same symmetry.  On Figure \ref{fig:LQ_2D}, we display snapshots of  $m$, $u$ and the optimal feedback law at several times. We see that the evolution is quite different from that displayed on Figure~\ref{fig:sym_cor_2D}, since the two parts of the distribution initially supported in two opposite corners of the domain symmetrically split into two parts each, which travels directly to the targets. The initial splitting phase is due to the coupling cost $cm(t,X_t)$ which favors the spread of the distribution. The path followed by the agents is  close to a shortest path to the targets.

\begin{center}
  \includegraphics[width=\linewidth]{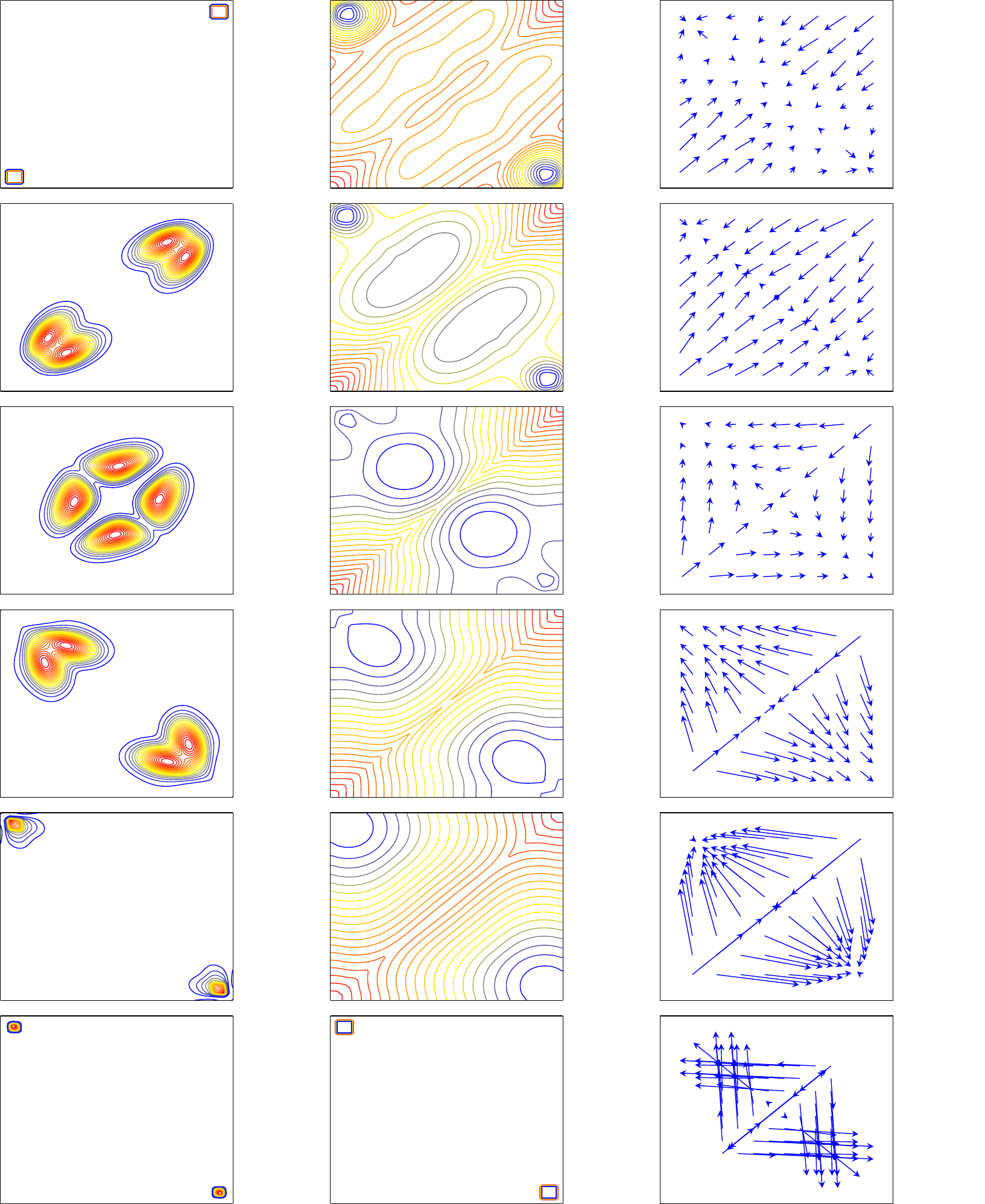}
  \captionof{figure}{Example 1. Gathering-kissing-splitting; snapshots  at  $t=0,0.2,0.4,0.6,0.8,1$. Left: contours of $m$. Center:  contours of $u$. Right: optimal feedback $\aa$.}
  \label{fig:sym_cor_2D}
\end{center}

\begin{center}
  \includegraphics[width=\linewidth]{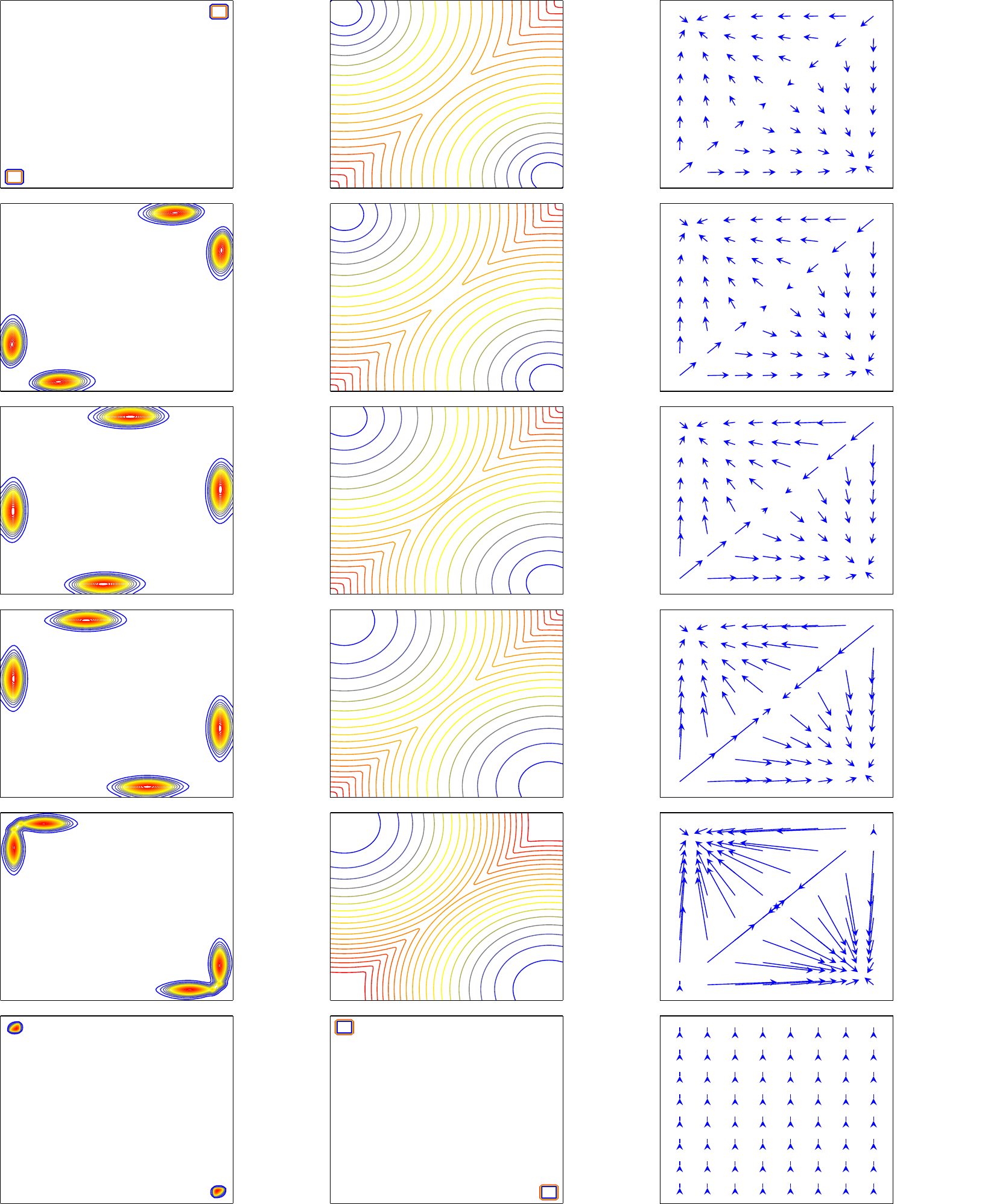}
  \captionof{figure}{Example 1. Same parameters except $\theta\lambda=0$; snapshots  at  $t=0,0.2,0.4,0.6,0.8,1$. Left: contours of $m$. Center:  contours of $u$. Right: optimal feedback $\aa$.}
  \label{fig:LQ_2D}
\end{center}

\subsubsection{Non-uniqueness of solutions}
\label{sec:non-uniq-solut}
The solution of the MFGC displayed on Figure~\ref{fig:sym_cor_2D} is likely to be unstable because 
the paths followed by the agents are significantly longer than the shortest paths to the targets. We expect that there are other solutions. We are going to see that this is indeed the case. 
For that purpose, we are going to introduce a vanishing perturbation of the initial distribution which breaks the symmetry of the problem. 
 This will lead to  different solutions to (\ref{eq:HJB})-(\ref{eq:BCNeum}). An example of  perturbation is displayed on Figure \ref{fig:evan_mass}. It consists of adding very little mass at the top and the bottom of the domain; the perturbed distribution is no longer symmetric with respect to the  diagonals. 
We expect that the agents initially distributed near  the  bottom target will go to the right, and that all the agents initially distributed at the bottom of the domain will follow them, because of the interaction through controls. Similarly, we expect that all the agents initially located at the top of the domain will go to the left.
\\
In our simulations, we use a continuation method, i.e. we consider perturbations corresponding to a decreasing sequence of nonnegative parameters $(\pi_n)_{n\in 0,\dots, N}$.  The last value $\pi_N=0$ corresponds to the distribution displayed on the left of Figure \ref{fig:CI_FC}.
For each new value $\pi_{n+1}$ of the parameter, we initialize the Newton iterations described in Section \ref{sec:newt-algor-solv} by the solution corresponding to the preceding value $\pi_{n}$. For positive values of $\pi_n$, the simulated solution is not symmetric with respect to the diagonals, and this property is preserved for the last value $\pi_{N}=0$.
\begin{center}
  \includegraphics[scale=0.5]{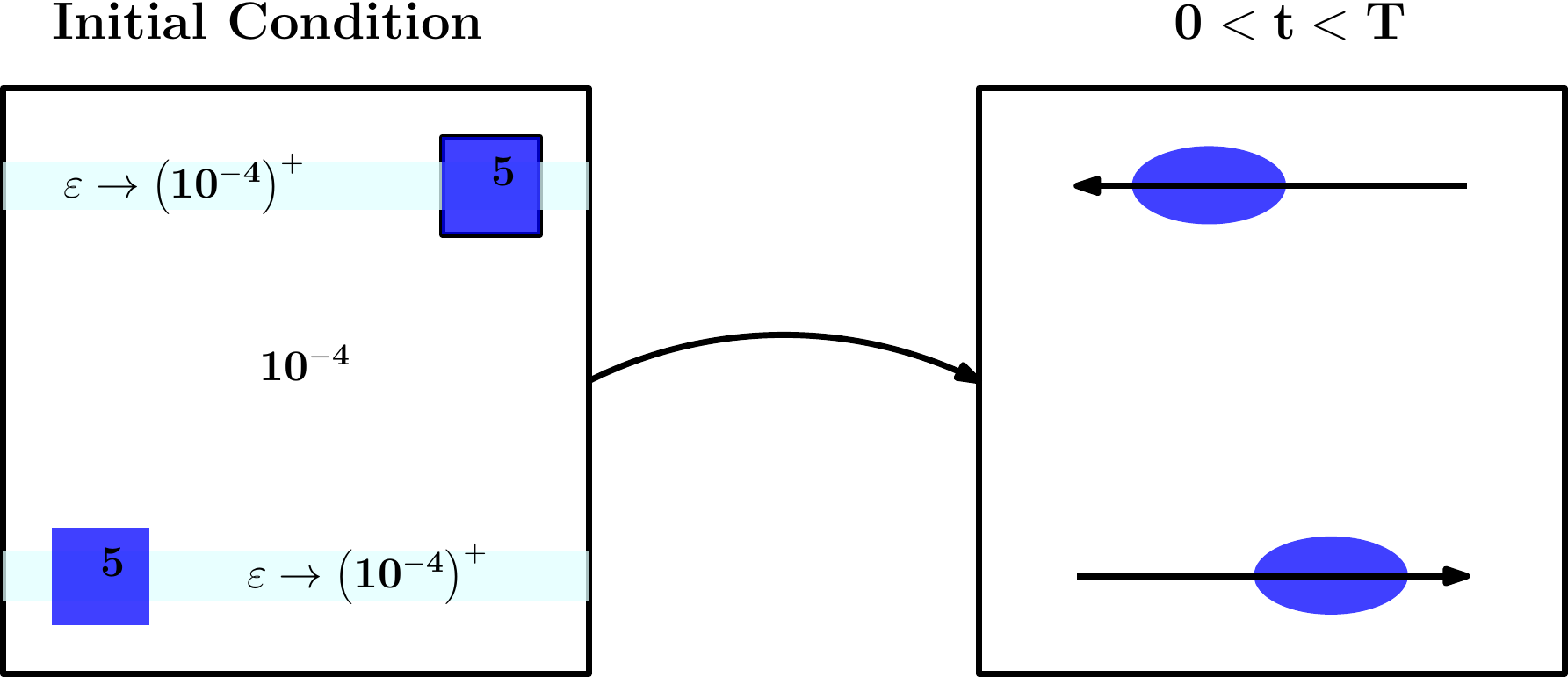}
  \captionof{figure}{Example 1. A continuation method leading to a different solution. Left:  a vanishing perturbation of the initial distribution. Right:  the expected evolution of the distribution.}
\label{fig:evan_mass}
\end{center}

On Figure~\ref{fig:non_uniq}, we display three different solutions (the distribution of states at different times) obtained with three different sequences of vanishing perturbations of the initial distribution displayed on Figure \ref{fig:CI_FC}: the solution displayed on the left corresponds to the sequence of perturbations displayed on Figure  \ref{fig:evan_mass}; applying to the pertubation a symmetry with respect to a diagonal,
 we obtain the solution displayed on the center; 
the solution displayed on the right is obtained by another kind of perturbation symmetric with respect to one diagonal only, located at the top and at the left of the domain: for this solution, we again notice gathering and kissing phases, and that the paths followed by the agents deviate from the shortest ones.

\begin{center}
  \includegraphics[width=\linewidth]{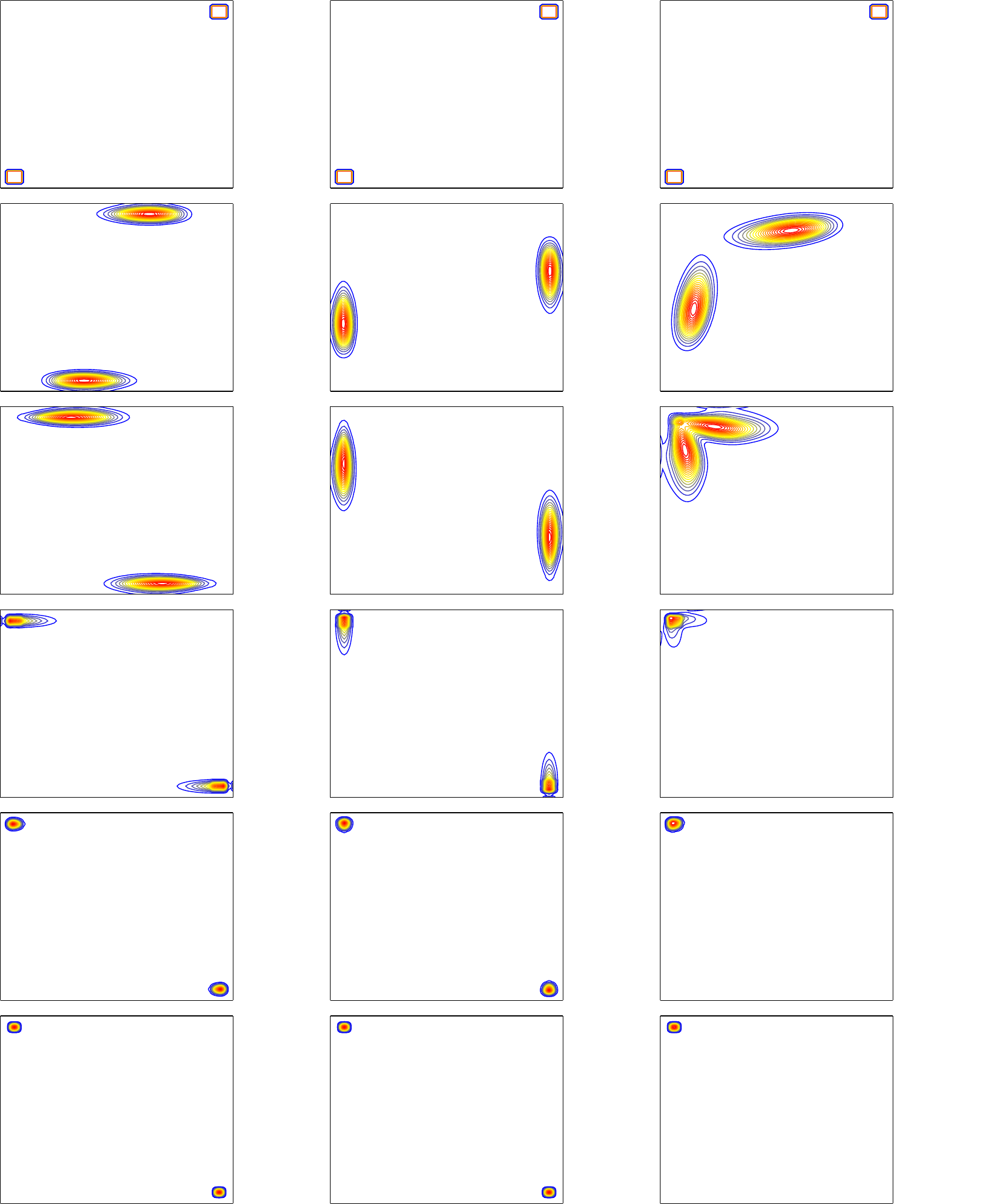}
  \captionof{figure}{Example 1. Different solutions obtained by adding  vanishing perturbations  to the initial distribution of states. The distribution  at time $t=0,0.2,0.4,0.6,0.8,1$.  The three columns correspond to three different sequences of perturbations.}
  \label{fig:non_uniq}
\end{center}
 
\subsubsection{Behaviour of the algorithm}
\label{sec:behaviour-algorithm}

In this paragraph, we investigate how the behaviour of the algorithm described in Section~\ref{sec:newt-algor-solv} is affected by  the variations of the 
parameters in the model. Recall that the algorithm is based on Newton iterations with a BiCGStab inner loop at each step. 

We start by the grid size and the viscosity parameter $\nu$: in the experiments reported below, 
 the stopping criterion for the outer Newton iteration is that the
 normalized Euclidean norm of the residual is smaller than $10^{-8}$.
 The  stopping criterion for the inner BiCGStab iterations is that the ratio of the norms
 of the current   and the initial residuals  is smaller than $10^{-7}$   (hence, it involves a relative error).  The parameters of the model are set to 
 \begin{displaymath}
    \th=1, \qquad \ll=0.9, \qquad c=10^{-3},\qquad  \rho=0.2,\qquad  L=1.
 \end{displaymath}
 Table   \ref{tab:1} displays  the iterations counts  for the  outer Newton  and the inner BiCGStab loops for different viscosities and grid sizes.
 We notice that the number of iterations is not very sensitive to the grid size, as expected because 
the map $\cG-I_d$,  with $\cG$ defined in (\ref{eq:20})~(\ref{eq:21}), is the discrete version of a map which has  compactness properties.
    The iteration counts increase  as $\nu$ is decreased.

 \begin{table}[htbp]
   \centering
   \caption{    Number of outer Newton iterations and average number of inner BiCGStab iterations per Newton step (the stopping criteria are given in the text)    with different viscosities and grid sizes. }
     \begin{tabular}[c]{|l|c|c||c|c||c|c||c|c|}
    \hline
    $\nu$
    &\multicolumn{2}{c|}{$26\times26\times26$}
    &\multicolumn{2}{c|}{$51\times51\times51$}
    &\multicolumn{2}{c|}{$76\times76\times76$}
    &\multicolumn{2}{c|}{$101\times101\times101$}
    \\
    \hline
    & Newton & BiCGStab
    & Newton & BiCGStab
    & Newton & BiCGStab
    & Newton & BiCGStab
    \\
    \hline
    $0.5$ & $3$ & $13.67$
    & $3$ & $12.67$ 
    & $3$ & $15$
    & $3$ & $14.33$
    \\
    \hline
    $0.1$ & $9$ & $14.67$
    & $10$ & $21.9$ 
    & $12$ & $15.08$
    & $12$ & $18.25$
    \\
    \hline
    $0.05$ & $4$ & $16.5$
    & $4$ & $17.25$ 
    & $4$ & $16.25$
    & $4$ & $19.25$
    \\
    \hline
    $0.01$ & $4$ & $42.5$
    & $4$ & $29$ 
    & $5$ & $30.4$
    & $5$ & $28.8$
\\
    \hline
   \end{tabular}
 \label{tab:1}
 \end{table}

Table \ref{tab:2}
 displays the iteration counts for different choices of $\ll$ and $\th$,
and  fixed values the other parameters:
\begin{displaymath}
   \nu=10^{-3},  \qquad c=10^{-3},\qquad  \rho=0.2,\qquad L=1,
\end{displaymath}
with a $101\times 101$ nodes grid and $101$ time steps.

In fact, the  numbers in Table~\ref{tab:2} are obtained by running a continuation method in $\theta$ and $\lambda$:
for each cell of the table, the solution corresponding either to the left
or the upper neighboring cell is used as an initial guess for the Newton iterations. If a cell has two such neighbors, the choice of the initial guess is made as follows: in the top-right triangular part of the table, strictly above the diagonal, we choose the initial guess corresponding to the left neighboring cell. In the bottom-left triangular part of the table, including the diagonal, we choose  the initial guess corresponding to the cell immediately above.
\\
 We see that changing $\ll$ and $\th$ mostly impacts  the number of iterations
of the inner loop. A reason for that is that $V$ is obtained by solving a  linear fixed point problem
 whose contraction factor is $\ll\th$. Hence, it is sensible that the number  
 of iterations necessary to solve the systems of linear equations  arising in the Newton steps increases as $\ll\th$ tends to $1$.
\begin{table}[htbp]
  \begin{center}
 \caption{\label{tab:2} Sensitivity to $\lambda$ and $\theta$.  }
    \begin{subtable}{0.45\textwidth}
      \begin{center}
        \subcaption{Average number of BiCGStab iterations.}
    \begin{tabular}{|l|c|c|c|c|c|}
        \hline
        \backslashbox{$\ll$}{$\th$} & $0.2$ 
        & $0.4$ & $0.6$ & $0.8$ & $1$
        \\
        \hline
        $0.2$ & $3.5$ & $4.5$
        & $7$ & $8$ & $8$
        \\
        \hline
        $0.4$ & $4.5$ & $7$
        & $8$ & $9$ & $10.67$
        \\
        \hline
        $0.6$ & $5$ & $8$
        & $9$ & $12$ & $15.67$
        \\
        \hline
        $0.8$ & $7$ & $8.67$
        & $10.67$ & $15.67$ & $27.25$
        \\
        \hline
        $0.9$& $7.33$ & $8.67$ 
        & $12$ & $18$ & $40$
        \\
        \hline
    \end{tabular}
  \end{center}
\end{subtable}
    \begin{subtable}{0.45\textwidth}
      \begin{center}
        \subcaption{number of Newton iterations.}
    \begin{tabular}{|l|c|c|c|c|c|}
        \hline
        \backslashbox{$\ll$}{$\th$} & $0.2$ 
        & $0.4$ & $0.6$ & $0.8$ & $1$
        \\
        \hline
        $0.2$ & $2$ & $2$
        & $2$ & $2$ & $2$
        \\
        \hline
        $0.4$ & $2$ & $3$
        & $2$ & $3$ & $3$
        \\
        \hline
        $0.6$ & $2$ & $3$
        & $3$ & $3$ & $3$
        \\
        \hline
        $0.8$ & $3$ & $3$
        & $3$ & $3$ & $4$
        \\
        \hline
        $0.9$ & $3$ & $3$
        & $3$ & $3$ & $6$
        \\
        \hline
    \end{tabular}
  \end{center}
\end{subtable}
  \end{center}
\end{table}

Recall that we use a continuation method, consisting of decreasing progressively $\nu$ until it reaches the desired value;
in Figure \ref{fig:BiCG_params} (respectively Figure \ref{fig:New_params}), we  plot the average number of BiCGStab iterations versus $\nu$,
(respectively the number of Newton iterations versus $\nu$)
 for different values of $\ll$, $\th$ and $c$. The stopping criteria have been described above. The grid contains $101\times 101$ nodes  and there are $101$ time steps.
 We recover the information contained in Tables   \ref{tab:1} and  \ref{tab:2}, i.e. that $\nu$  mostly impacts the number of iterations in the inner loop.  We also notice that when $\nu$ is large, the number of iterations seems to depend more on  $\ll$ and $\theta$ than on $c$ and that it becomes highly sensitive to $c$ when $\nu$ is small, but we do not really know how to explain this. The number of iterations in the outer Newton loop seems much less sensitive than the number of inner BiCGStab iterations. Looking at the overall number of BiCGStab iterations, they are not very different for the choices $(\ll,\th,c)=(0.9,1,0)$ and $(\ll,\th,c)=(0,0,0.1)$, since  the latter case needs more Newton iterations but less BiCGStab iterations per Newton steps.

 \begin{center}
   \includegraphics[scale=0.7]{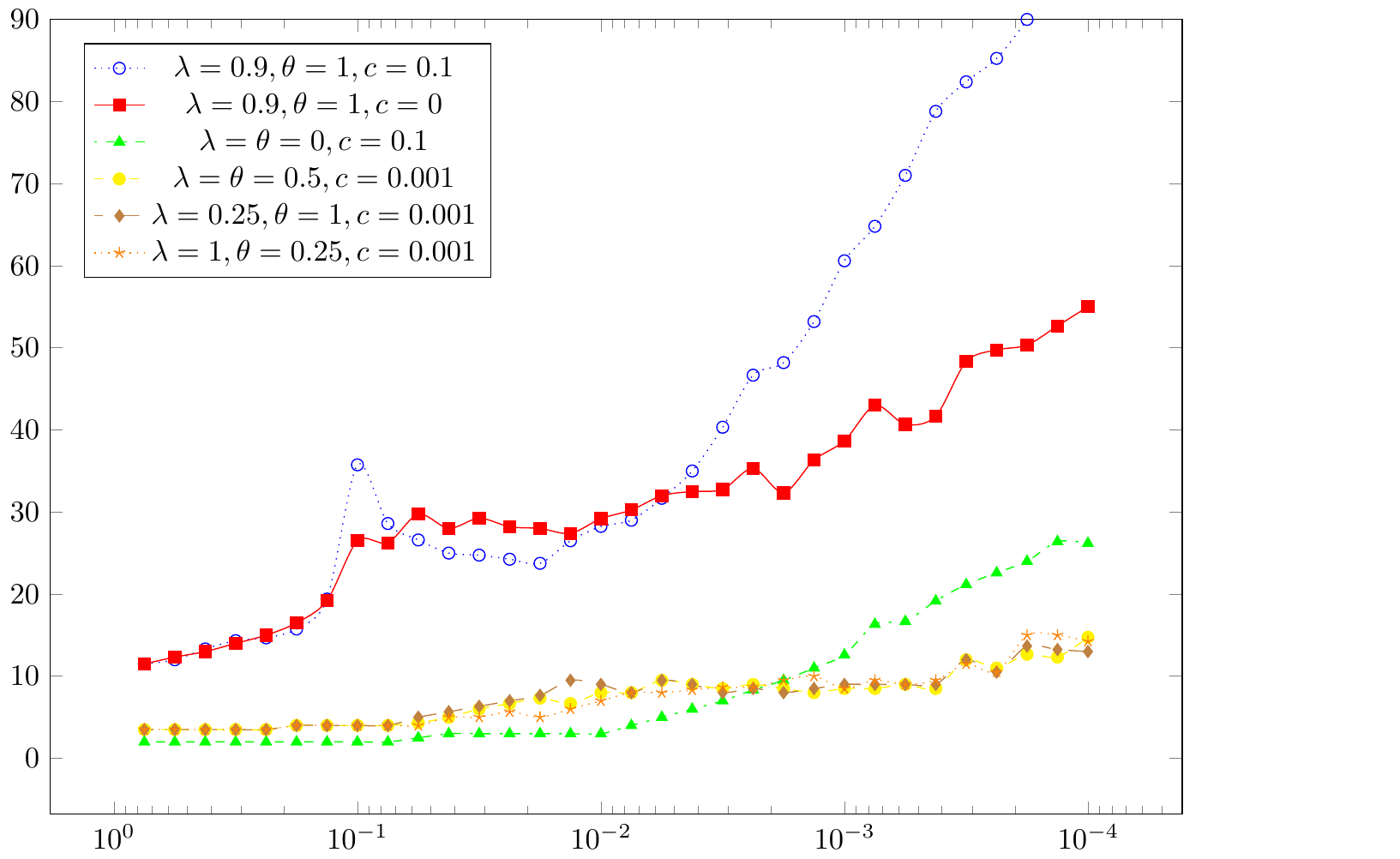}
   \captionof{figure}{\label{fig:BiCG_params}    Average number of iterations of BiCGStab iterations per Newton step  versus $\nu$ for different choices of $\lambda$, $\theta$ and $c$.  }
   \end{center}

   \begin{center}
       \includegraphics[scale=0.7]{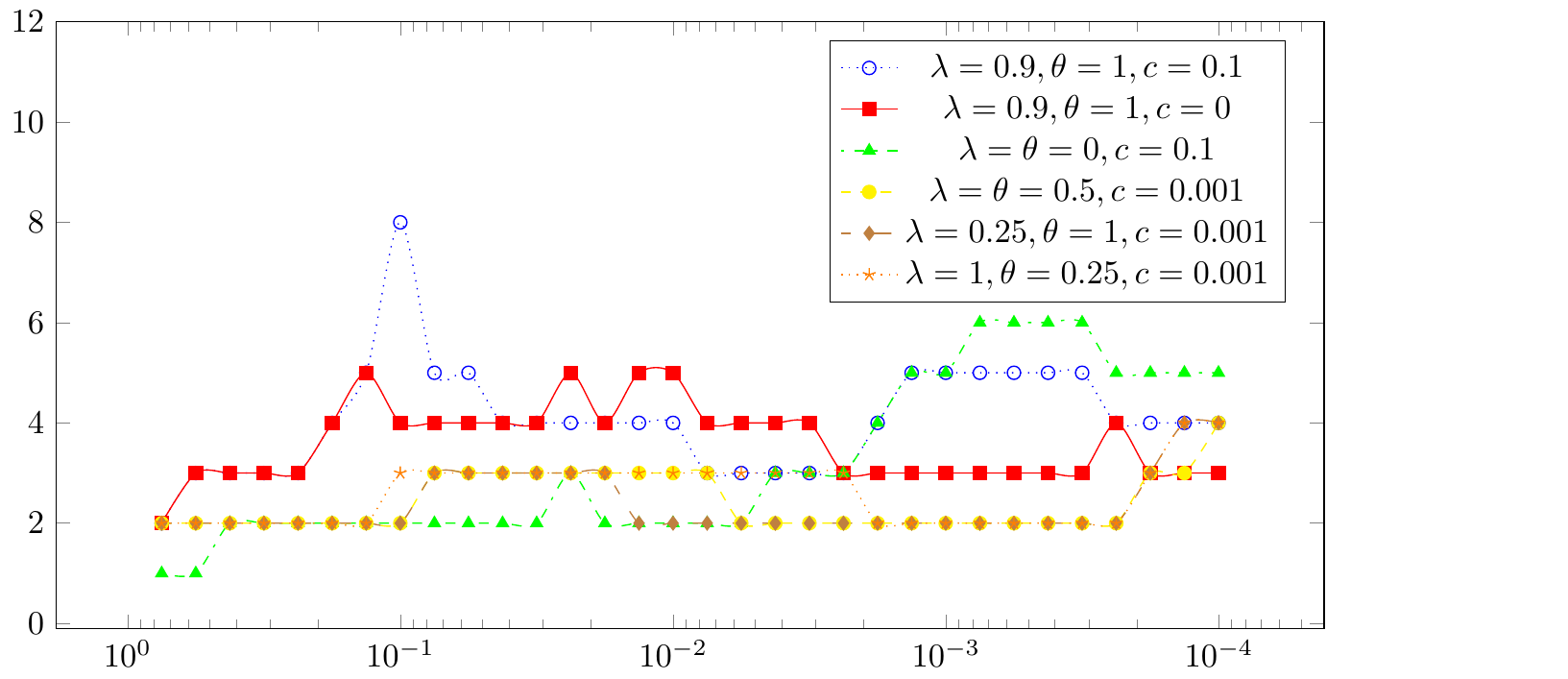}
    \captionof{figure}{\label{fig:New_params} Number of iterations of  Newton iterations  versus $\nu$ for different choices of $\lambda$, $\theta$ and $c$.    }
\end{center}

Finally, we investigate the sensitivity of the algorithm to  $L$, the number of fixed point iterations yielding a proxy of $V$, see
paragraph~\ref{sec:coupl-terms-aver}.
Table \ref{tab:3} contains  the iteration counts of the inner and outer loops for different values of
$\nu$ and $L$, and two choices
 of $(\lambda,\theta)$ (all the other parameters are fixed
$c=10^{-3}$, $\rho=0.2$).  We notice that the choice of $L$ seems to have little impact on the iteration count.  For these reasons, $L=1$ appears to be a good choice.

\begin{table}[htbp]
   \centering
   \caption{    Number of outer Newton iterations and average number of inner BiCGStab iterations per Newton step (the stopping criteria are given in the text)    with different viscosities and $L$.}
     \begin{tabular}[c]{|l|c|c||c|c||c|c||c|c|}
         \hline
         \multicolumn{9}{|c|}{$\ll=\th=0.8$}
         \\
    \hline
    \backslashbox{$\nu$}{$L$}
    &\multicolumn{2}{c|}{$1$}
    &\multicolumn{2}{c|}{$2$}
    &\multicolumn{2}{c|}{$3$}
    &\multicolumn{2}{c|}{$5$}
    \\
    \hline
    & Newton & BiCGStab
    & Newton & BiCGStab
    & Newton & BiCGStab
    & Newton & BiCGStab
    \\
    \hline
    $0.1$ & $3$ & $7.33$
    & $3$ & $6$ 
    & $3$ & $6$
    & $3$ & $5.33$
    \\
    \hline
    $0.01$ & $3$ & $13.67$
    & $3$ & $13$ 
    & $3$ & $15.33$
    & $3$ & $13.67$
    \\
    \hline
    $0.001$ & $2$ & $14.5$
    & $2$ & $14.5$ 
    & $2$ & $17.5$
    & $2$ & $16$
    \\
    \hline
    $0.0001$ & $3$ & $22$
    & $3$ & $22.67$ 
    & $3$ & $22.33$
    & $3$ & $24.67$
\\
    \hline
         \multicolumn{9}{|c|}{$\ll=0.9$, $\th=1$}
         \\
    \hline
    \backslashbox{$\nu$}{$L$}
    &\multicolumn{2}{c|}{$1$}
    &\multicolumn{2}{c|}{$2$}
    &\multicolumn{2}{c|}{$3$}
    &\multicolumn{2}{c|}{$5$}
    \\
    \hline
    & Newton & BiCGStab
    & Newton & BiCGStab
    & Newton & BiCGStab
    & Newton & BiCGStab
    \\
    \hline
    $0.1$ & $4$ & $26.75$
    & $4$ & $19.75$ 
    & $4$ & $19.5$
    & $4$ & $19.25$
    \\
    \hline
    $0.01$ & $4$ & $29$
    & $4$ & $25.75$ 
    & $5$ & $28$
    & $5$ & $29$
    \\
    \hline
    $0.001$ & $3$ & $36.67$
    & $3$ & $34$ 
    & $3$ & $38.67$
    & $3$ & $37.67$
    \\
    \hline
    $0.0001$ & $3$ & $55$
    & $3$ & $58.33$ 
    & $3$ & $64$
    & $3$ & $63.33$
    \\
    \hline
   \end{tabular}
 \label{tab:3}
 \end{table}


\subsection{Second example}\label{sec:second-example}

As a  second example, we consider a model for a crowd crossing a  hall, with an entrance at the left and an exit at the right.
Roughly speaking, a generic agent will interact with those located in a cone ahead of her. 

\subsubsection{Description of the model}
\label{sec:description-model-3}

The state space is the rectangle $\Omega=(-1,1)\times (-0.1,0.1)$ and the time horizon is $T=8$.  Let  $\ell$ denote the length of the bottom and upper sides of $\partial \Omega$: $\ell=2$.

The boundary of $\Omega$ is  split into three parts $\Gamma_{ND}$, $\Gamma_{NN}$ and $\Gamma_{DD}$   which respectively stand for an entrance, some walls and an exit:
\begin{equation}
  \label{eq:31}
   \Gamma_{DD}=  \{1\}\times  [-0.05,0.05],\qquad    \Gamma_{ND}=   \{-1\}\times  [-0.05,0.05] , \qquad       \Gamma_{NN}=\partial\O\backslash\lp\Gamma_{DD}\cap\Gamma_{ND}\rp,
\end{equation}
and we consider the following boundary conditions:
\begin{eqnarray}
\label{eq:32}
\left.   \begin{array}[c]{rcr}
     \frac{\partial u}{\partial n}(t,x)&=&  0\\
     \nu\frac{\partial m}{\partial n}(t,x) -\ll\th  m(t,x) V(t,x)\cdot n(x)&=&0 
  \end{array}
\right\} \qquad &\text{on } &\Gamma_{NN}, \\
\label{eq:33}
\left.   \begin{array}[c]{rcr}
     u(t,x) &=& 6 \\
\nu \frac {\partial m }{\partial n}(t,x)+ m(t,x)H_p(\nabla u(t,x), V(t,x)))\cdot n(x)&=& -2
  \end{array}
\right\} \qquad &\text{on } &\Gamma_{ND}, \\
\label{eq:34}
\left.   \begin{array}[c]{rcr}
     u(t,x) &=& -4\\
     m(t,x)&=&  0
  \end{array}
\right\} \qquad &\text{on } &\Gamma_{DD},
\end{eqnarray}
where $H$ is defined below.

The system of PDEs is still (\ref{eq:HJB})-(\ref{eq:defZ}), with the following new features:
\begin{itemize}
\item compared to (\ref{eq:1}), the part ot the running cost accounting for the interaction through controls is multiplied by a positive normalization factor $a$:
\begin{equation}
    \label{eq:defL}
    \begin{aligned}
        L(\aa,V)=        a\frac{\th}2|\aa-\ll V|^2        +a\frac{1-\th}2|\aa|^2,
    \end{aligned}
\end{equation}
and the associated Hamiltonian by the Legendre's transform is
\begin{equation}\label{eq:30}
    H(p,V)
    =
    \frac{a}{2}
    \labs\frac{p}a-\ll\th V\rabs^2
    -a\frac{\ll^2\th}2|V|^2.
\end{equation}
We  choose the  normalization factor  $a$ as follows:
 \begin{equation}
\label{eq:6}    a    = \ati     (1-\ll^2\th)^{-1},
\end{equation}
where $\ati$ is a positive constant (independent of $\lambda$ or $\theta$). As we shall see later, the normalization allows us to compare the solutions 
obtained with different values of  $\ll$ and $\th$. We shall also see that this normalization is  specific to the example under consideration and may not be relevant in other situations.
\item The kernel $K$ is given by 
  \begin{equation}\label{eq:35}
    K(x,y)=  \mathbf{1}_{y_1\geq x_1} K_\rho(|x-y|) \kappa( |\omega| ),
  \end{equation}
where $\omega$ is the angle made by the vector $y-x$ with the vector $(1,0)$ and
\begin{enumerate}
\item $k_\rho$ is defined  as in the first example
\item $\kappa$ is a nonincreasing $C^1$   function defined  such that 
  \begin{displaymath}
    \kappa=\left\{
      \begin{array}[c]{lcr}
        1,\quad & \text{if }& 0\le \omega\le 0.9 \,\omega_0,\\
        0,\quad & \text{if }& \omega \ge  \omega_0,
      \end{array}
 \right.
  \end{displaymath}
\end{enumerate}
for a given angle $\omega_0\in (0,\frac \pi 2)$.
\item The parameter $c$ in (\ref{eq:HJB}) is chosen to be zero
\item We shall consider various functions $f_0$ modeling the cost for staying at a given point in the domain.
\end{itemize}
The initial condition is $m_0(x)=10^{-4}$ for all $x\in \Omega$, and the terminal cost is $0$.

\subsubsection{The case when $f$ is constant: comparisons with a  one-dimensional problem}
\label{sec:simple-case-related}
In this paragraph, we choose $f_0(x)=F$ for all $x\in \Omega$. 
\paragraph{Simplification: a one-shot one-dimensional game}
\label{sec:simpl-one-dimens}
We approximate the MFGC by a  one-shot one-dimensional game.  The state space is the interval $[0,\ell]$. The points $x=0$ and $x=\ell$ respectively stand for an entrance and  an exit. When she enters the domain, a representative agent chooses her drift $\alpha\in \RR_+$ once and for all, and her dynamics is given by  $  x'(t)=\aa$ for $t= \ell/\alpha$. The distribution of drifts is a probability measure  $\mu$ on $\RR_+$ and the average drift is $V=\int_{\alpha\in \RR_+} \alpha d\mu(\alpha)$. There is no entry or exit costs.
\\
With  the running cost $L(\alpha,V)+F$, where $L$ is given by (\ref{eq:defL}), the mean field Nash equilibrium reads
\begin{displaymath}
  \text{support}({\mu})\subset \text{argmin}_{\alpha}   \frac {\ell}{\alpha} 
\left(        a\left(\frac{\theta}2(\alpha-\lambda V)^2        +\frac{1-\theta}2 \alpha^2 \right) +F \right) .
\end{displaymath} 
 Given $V$, one checks that the unique solution of the minimization problem in $\alpha$ is 
 \begin{equation}
   \label{eq:39}
\alpha^*(V) = \left( \frac {2F+a\theta \lambda^2 V^2} a \right)^{\frac 1 2 }.
 \end{equation}
Hence $\mu= \delta_{\alpha ^*(V)}$, and  the static mean field  Nash equilibrium reads
 \begin{displaymath}
   V= \left( \frac {2F+a\theta \lambda^2 V^2} a \right)^{\frac 1 2 },
 \end{displaymath}
which yields 
\begin{equation}
  \label{eq:36}
\alpha^*= V= \left( \frac {2F} {a (1-\theta \lambda^2)}\right)^{\frac 1 2 }.
\end{equation}
This leads us to choose $a=\widetilde a / (1-\theta \lambda^2)$, which yields $\alpha^*= V= \sqrt{\frac {2F } {\widetilde a}}$.
Note that the value corresponding to the mean field Nash equilibrium is 
\begin{equation}
\label{eq:37}  
        J_{\text{MFG}} =  \frac {\ell}{\alpha^*} 
\left(        a\left(\frac{\theta}2(\alpha^*-\lambda \alpha^*)^2        +\frac{1-\theta}2 (\alpha^*)^2 \right) +F \right) 
       =
        \ell\frac{1-\lambda \theta}{1-\lambda^2\theta}\sqrt{2F\ati}.
\end{equation}

\begin{remark}
  \label{sec:simpl-one-shot-1}
Note that if $\theta=\lambda=1$, then (\ref{eq:39}) implies that $\alpha^* (V)>V$, while an equilibrium corresponds to $V=\alpha^*(V)$. Thus, there is no mean field equilibrium if $\theta=\lambda=1$.
\end{remark}
\begin{remark}[Comparison with the one-shot mean field type control problem]
  \label{sec:simpl-one-shot}
The Nash equilibrium must be distinguished from the situation in which
the drift of the agents is determined in order to minimize the global cost; the latter problem belongs to the class of {\sl mean field type control problems} (MFTC for short)) studied  by A. Bensoussan and his collaborators, \cite{MR3134900} and R. Carmona and F. Delarue, \cite{MR3395471}. \\
In this case, the problem reads
\begin{displaymath}
\inf_{\mu}  \int_{\alpha \in \RR_+}      \frac {\ell}{\alpha} 
\left(        a\left(\frac{\theta}2\left (\alpha-\lambda  \int_{\beta \in \RR_+} \beta d\mu(\beta)\right)^2        +\frac{1-\theta}2 \alpha^2 \right) +F \right)   d\mu (\alpha).
\end{displaymath}
Looking for the solution as the Dirac mass $\delta_{\alpha^{**}}$, we see that $\alpha^{**}$  minimizes 
\begin{displaymath}
  \alpha\mapsto    \frac {\ell}{\alpha}  \left(        a\left(\frac{\theta}2\left (\alpha-\lambda  \alpha \right)^2       
 +\frac{1-\theta}2 \alpha^2 \right) +F \right),
\end{displaymath}
thus
\begin{displaymath}
  \alpha^{**}= \left(    \frac {2F}  {a( 1-2\theta\lambda +\theta\lambda^2)}\right)^{\frac 1 2 }.
\end{displaymath}
The value of the problem is 
    \begin{equation*}
            J_{\text{MFTC}}
            =
            \ell\sqrt{2Fa(1-\lambda\theta(2-\lambda))}
            \leq
            \frac{\ell}2\sqrt{2Fa}
            \frac{1-\lambda \theta (2-\lambda)}{\sqrt{1-\lambda^2\theta}}
            +\frac{\ell}2\sqrt{2Fa(1-\lambda^2\theta)}            
            =J_{\text{MFG}},
    \end{equation*}
from Young's inequality $y\leq \frac{y^2}{2z}+\frac{z}2$, with $y=\sqrt{1-\lambda\theta(2-\lambda)}$    and $z=\sqrt{1-\lambda^2\theta}$.
 The value of the MFTC is lower than the value of the MFG,
 which is  natural since in the former case, the agents are collaborating to  minimize the global cost.
Moreover, the values of the two problems coincide only if 
$\lambda \theta=0$ , i.e. when there are no interactions through controls.
\end{remark}

\paragraph{Comparison with the one-shot one-dimensional problem}
We have made numerical simulations of the MFGC described in paragraph \ref{sec:description-model-3} with $f_0(x)=F=1$ for all $x\in \Omega$, $a=\ati(1-\ll^2\th)^{-1}$ and $\ati=2$. We wish to compare the results with the explicit formulas obtained above for the one-shot one-dimensional MFG. If the latter problem is a good approximation of the former one,  we should find  $\alpha_1$  close to $ \sqrt{\frac{2F}{\ati}}=1$, therefore mostly independent of $\lambda$ and $\theta$. 
Note that the approximation by the one-shot one-dimension mean field game is sensible only if the time remaining to the horizon is large enough, i.e. significantly larger than $\ell/\alpha^*=\ell$. If $T-t$ is small, then the optimal strategy for a representative player located on the left part of the domain is not to move.

On Figure \ref{fig:comp1D}, we display the norm of the optimal feedback for several values of $\lambda$ and $\theta$.
We see that the optimal drift is close to $1$ in agreement with the explicit formula found for the simplified  one-shot one-dimensional MFG.
We also see that the agents located at the bottom and top of $\Omega$ head toward the center of the domain (i.e. $x_2=0$) before reaching 
the exit. Thus, the  second coordinate of $\alpha$  gets large and there are singularities on both sides of the exit, whose amplitude increases
 as $\lambda\theta$ tends to $1$.

\begin{center}
    \includegraphics[width=\linewidth]{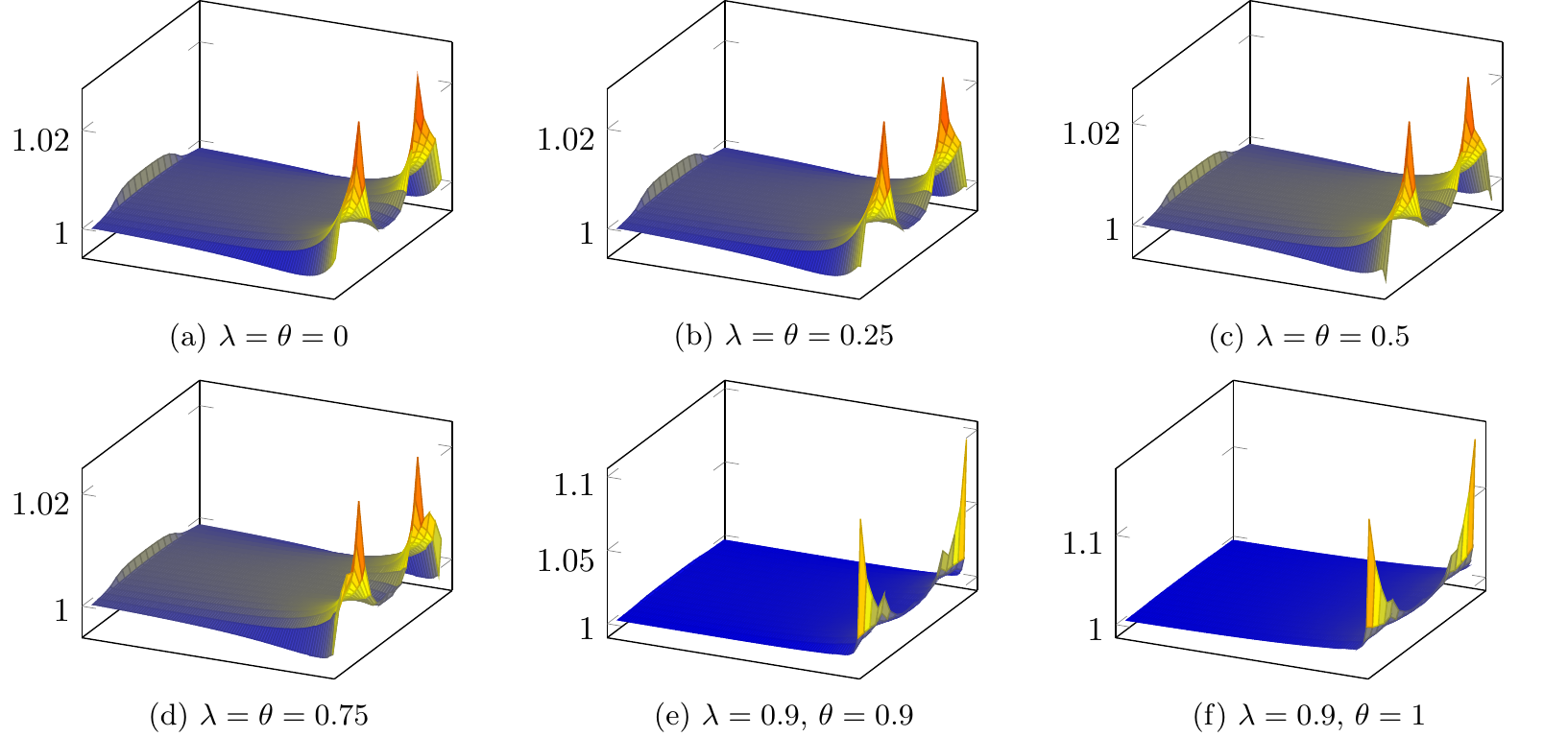}
    \captionof{figure}{     \label{fig:comp1D}   Example 2, $f_0(x)=1$.  The optimal feedback $\aa$    for different choices of $\ll$ and $\th$ at $t=T/2$.
    }
\end{center}

\subsubsection{Queues}
\label{sec:queuing-effects}
Here, we choose $f$ as follows:
\begin{displaymath}
    f_0(x) =    \lc
    \begin{aligned}
        4 ,\quad&\text{ if } \quad x\leq -0.1,
        \\
        1 ,\quad&\text{ if } \quad x\geq 0.1,
        \\
        2.5-15x, \quad  &\text{ otherwise};
    \end{aligned}
    \right.
\end{displaymath}
hence,
an agent pays a cost for staying in $\Omega$ which  is higher in the left part of the domain. 

The parameters are 
\begin{displaymath}
 \nu=0.001,\qquad \tilde a=2,\qquad   \ll=0.95, \qquad\th=0.95,\qquad \rho=0.25, \qquad \omega_0=\frac{\pi}3. 
\end{displaymath}
The grid has $101 \times 21$ nodes and there are $101$ time steps.

The  evolution of the distribution of states and of the optimal feedback is displayed on Figure \ref{figure:queue}. We compare these results with a simulation of the MFG  obtained by cancelling either $\theta$ or $\lambda$  while keeping all the other parameters and the grid unchanged, see Figure \ref{fig:queueLQ}.  On Figure \ref{figure:queue}, we see that  a well distributed queue takes place from the entrance to the exit, by contrast with 
 Figure \ref{fig:queueLQ} where we see that the agent rush and accumulate in the right part of the domain. Similarly, the deceleration is much stiffer in the latter case. It is not surprising that the interactions through controls have the effect of smoothing the distribution of states and the optimal feedback law.
On the bottom of Figure \ref{fig:queueLQ}, we see that when $t$ is close to the horizon, the distribution is mainly concentrated near the middle of the domain but slightly on the right: this corresponds to the agents that have reached the zone where $f_0$ has the smaller value, i.e. $1$, but for which reaching the exit before $T$ becomes too costly. There is also a smaller bump near the entrance corresponding to agents that would pay  too high a cost to reach the right part of the domain before $T$. These phenomena are clearly a side effects due to the finite horizon. These side effects are also present on  Figure \ref{figure:queue}, but 
they are attenuated  by the  interaction through the controls.  
 \begin{remark}
   \label{sec:queues}
 In models accounting for congestion effects, the cost of motion depends on the density of the distribution of states and gets higher in the crowded regions.
We refer to \cite{Lions_video} for a pioneering discussion of MFG models including congestion, to \cite{MR3765549} for the  analysis of the system of PDEs that arise with these models, and to \cite{MR3821977} for numerical simulations. Such models also permit to describe queueing phenomena, because the agents located in crowded regions pay a large cost for moving. In \cite{MR3821977}, since the congestion effects are taken into account in a local manner, the queues take place in small regions of the state domain.  By contrast,  the present nonlocal model accounts for the fact that the  agents anticipate low speed regions, making the traffic more fluid and the distribution of state smoother. Although
 it is quite possible to do, we have not incorporated congestion effects in the present model. We plan to do it in forthcoming works. 
\end{remark}

\begin{center}
    \includegraphics[width=\linewidth]{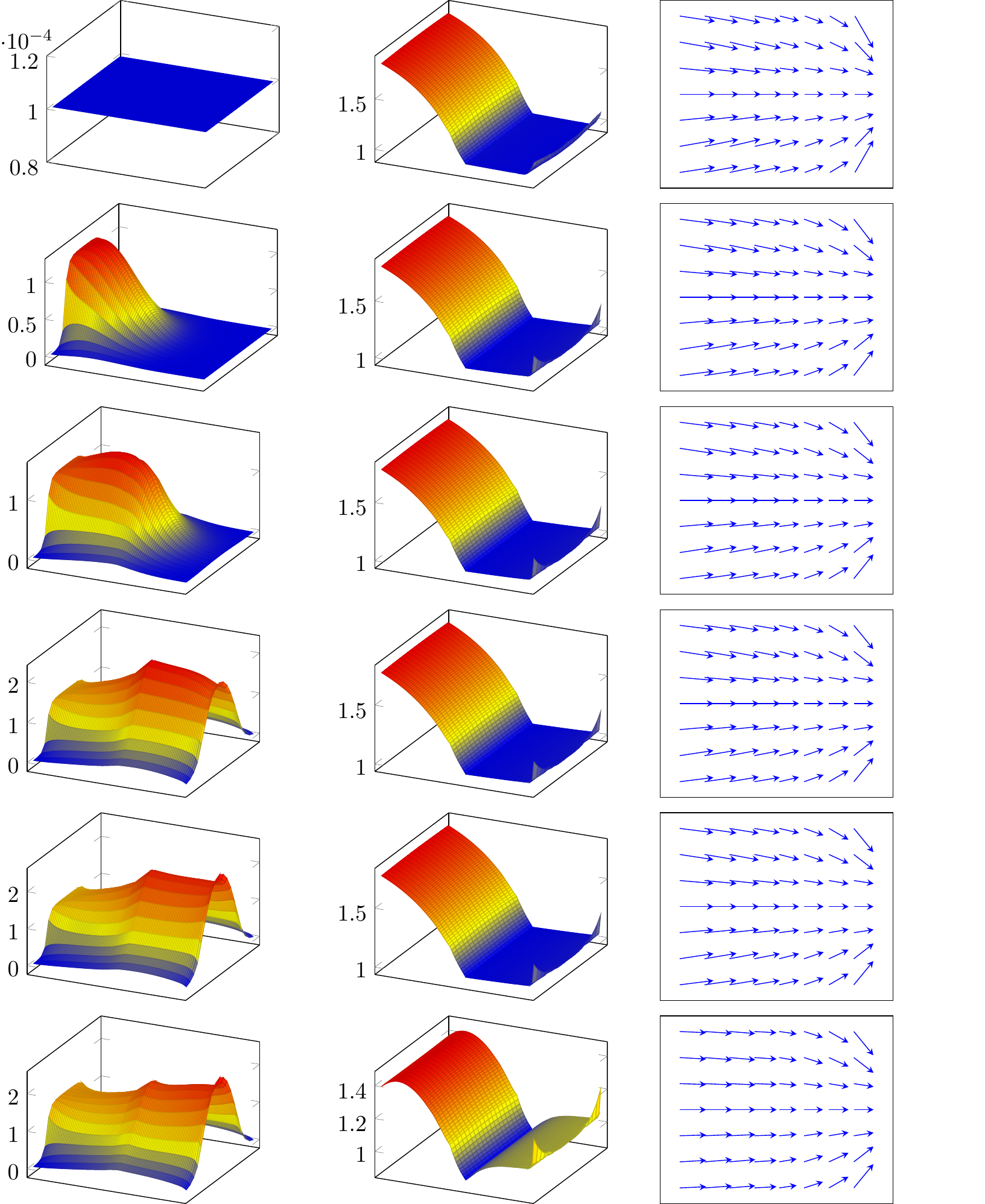}
    \captionof{figure}{ \label{figure:queue}
        Example 2.      Interaction via controls:  snapshots at  $t=0, 0.4, 0.8, 2, 4, 7$.      Left: the distribution   of states. Center: the norm of the optimal feedback. Right: the optimal feedback.    }
\end{center}

\begin{center}
    \includegraphics[width=\linewidth]{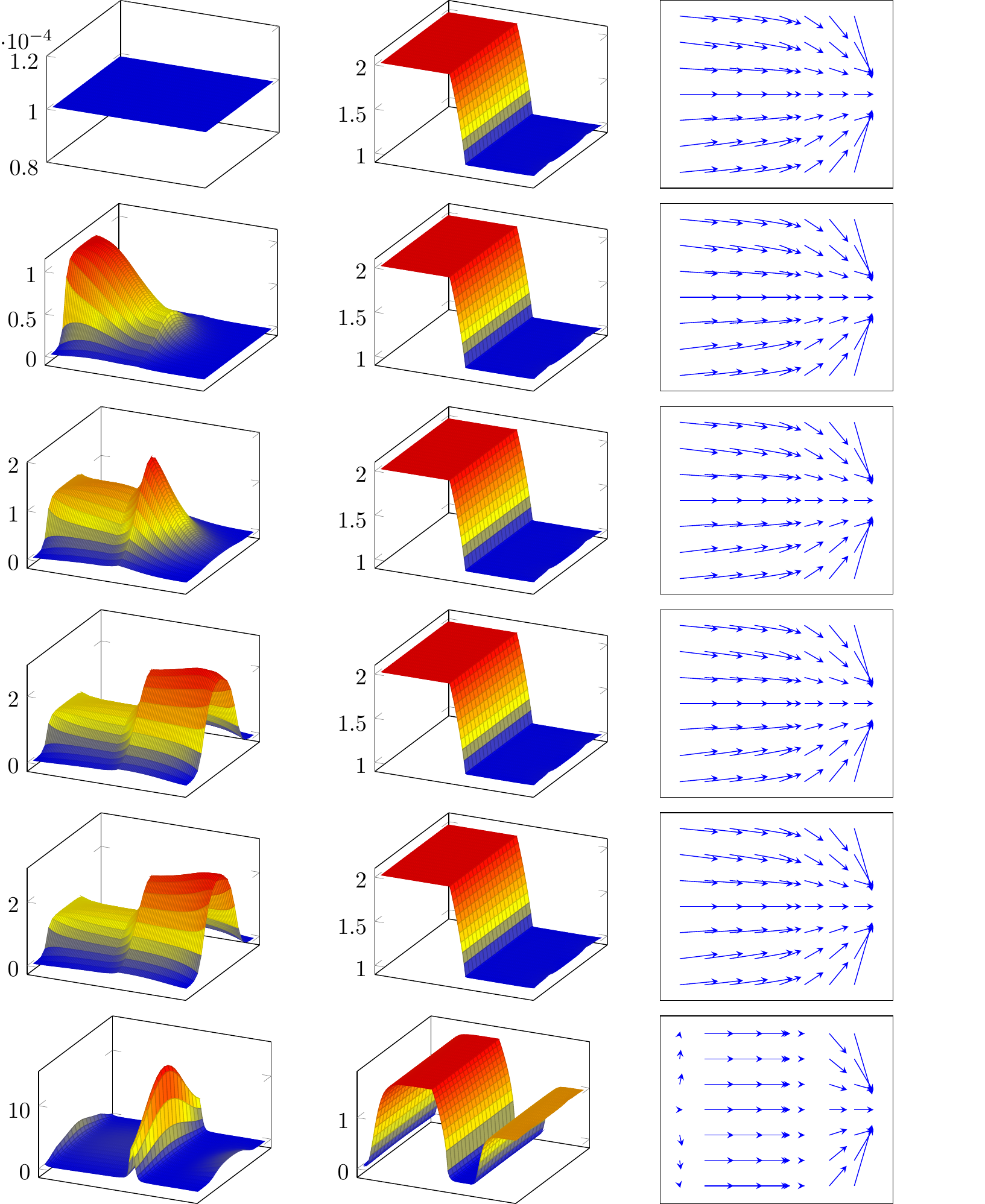}
    \captionof{figure}{\label{fig:queueLQ}      Example 2.      Same parameters except $\lambda \theta=0$:  snapshots at  $t=0, 0.4, 0.8, 2, 4, 7$.      Left: the distribution   of states. Center: the norm of the optimal feedback. Right: the optimal feedback. }
    \end{center}
    
\subsubsection{Stationary regime}
\label{sec:stationary-solutions}
We now look for a stationary equilibrium. For that, we use an iterative method in order to progressively diminish the effects of the initial and terminal conditions:  starting from $(u^0, m^0, V^0)$, the numerical solution of the finite horizon problem described above, we construct a sequence of approximate solutions $(u^\ell, m^\ell, V^\ell)_{\ell \ge 1}$ by the following induction:  $(u^{\ell+1}, m^{\ell+1}, V^{\ell+1})$ is the solution of the finite horizon problem with the same system of PDEs in $(0,T)\times \Omega$, the same boundary conditions on $(0,T)\times \partial \Omega$, and the  new initial and terminal conditions as follows:
\begin{eqnarray}
\label{eq:38}
u^{\ell+1}(T, x)&=& u^{\ell}\left(\frac T 2, x\right),\qquad x\in \Omega  ,\\
m^{\ell+1}(0, x)&=& m^{\ell}\left(\frac T 2, x\right),\qquad x\in \Omega  .
\end{eqnarray}
As $\ell$ tends to $+\infty$, we observe that  $(u^\ell, m^\ell, V^\ell)$ converge to time-independent functions. At the limit, we obtain a steady solution of 

\begin{eqnarray*}
        &
        - \nu\Delta u
        + \frac a 2 \labs \frac 1 a \nabla_xu-\ll\th V\rabs^2
        -a \frac{\ll^2\th}2\labs V\rabs^2
        = cm+f_0(x),\quad &\text{ in }
            \Omega,
        \\
        & 
        - \nu\Delta m
        -\divo\lp \lp \frac 1 a \nabla_xu
        -\ll\th V\rp m\rp
        =
        0,\quad &\text{ in }
            \Omega,
        \\
        &\ds V(x)
        =
        -    \frac 1 {Z(x)}\int_{\Omega}
        \lp \frac 1  a \nabla_xu(y)-\ll\th V(y)\rp K(x,y)
dm(y),  &\text{ in }
            \Omega,
        \\
        &Z(x)
        =
        \int_{\Omega}
        K(x,y)dm(y),
 &\text{ in }
            \Omega,
          \end{eqnarray*}
          
with the boundary conditions on $\partial\Omega$  given by (\ref{eq:32})-(\ref{eq:34}).

\begin{center}
    \includegraphics[width=\linewidth]{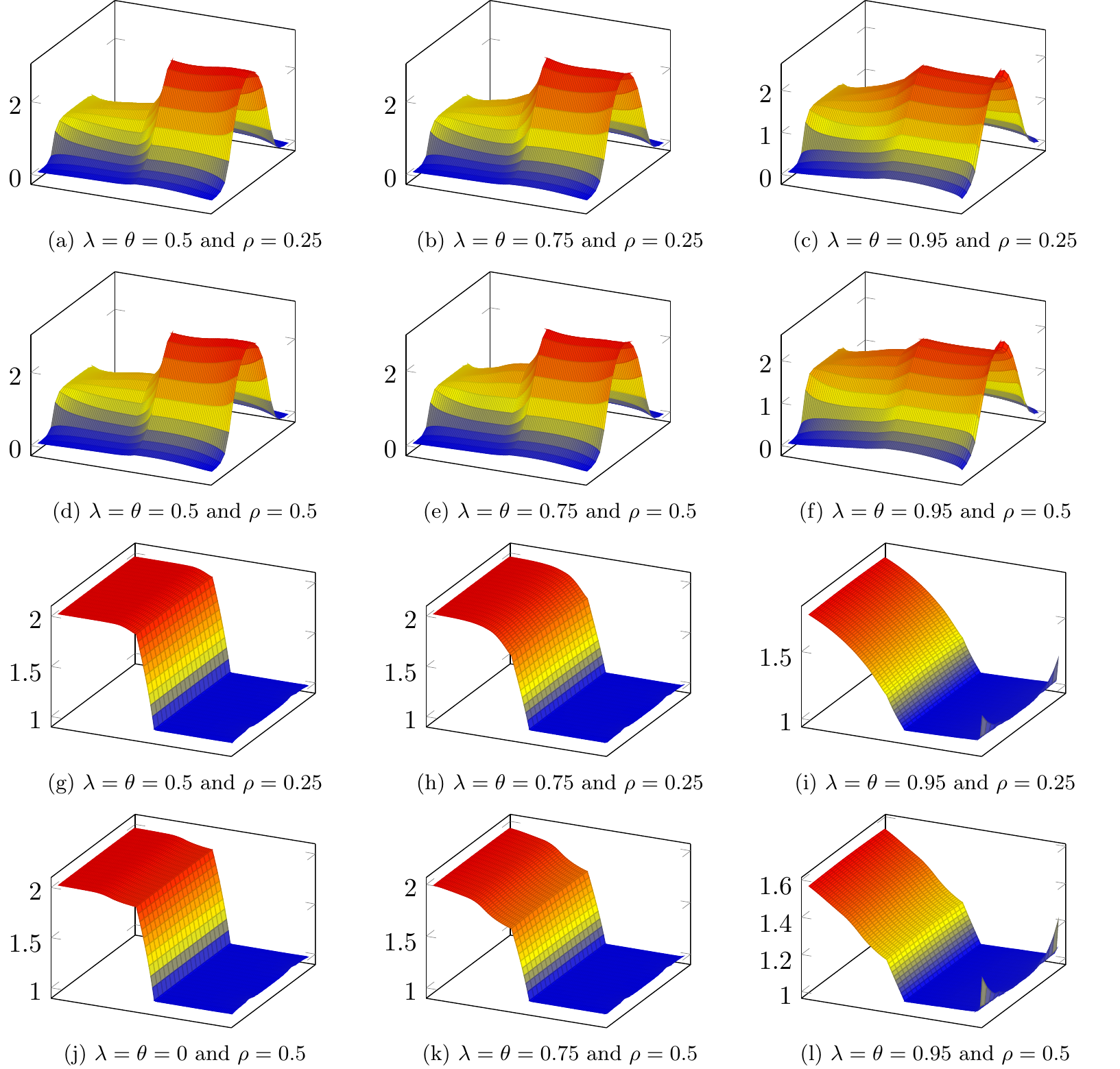}
    \captionof{figure}{\label{fig:stat}}
    Example 2.  Solutions of the stationnary problem for different choices
    of $\ll,\th$ and $\rho$.      Top six figures:  the distribution of states $m$.   Bottom six figures: norm of the optimal controls.
    \end{center}


\begin{acknowledgement}
  This research was partially supported by the ANR
  (Agence Nationale de la Recherche) through
MFG project ANR-16-CE40-0015-01.
\end{acknowledgement}

\bibliographystyle{plain}
\bibliography{MFGbiblio}

\end{document}